\documentclass[12pt]{article}
\usepackage[colorlinks=true,linkcolor=blue,citecolor=red, pdfborder={0 0 0}]{hyperref}
\usepackage{authblk}
\usepackage{amsmath}
\usepackage{amsfonts,dsfont,amsthm,amssymb,enumerate}
\usepackage{graphicx,fullpage}
\usepackage{mathtools}
\theoremstyle{plain}
\newtheorem{theorem}{Theorem}[section]
\newtheorem{lemma}[theorem]{Lemma}
\newtheorem{proposition}[theorem]{Proposition}
\newtheorem{corollary}[theorem]{Corollary}

\theoremstyle{definition}

\theoremstyle{remark}
\newtheorem{remark}[theorem]{Remark}

\numberwithin{equation}{section}
 
\DeclareMathOperator{\grad}{grad}

\newcommand{\CC}{\mathbb{C}}

\newcommand{\RR}{\mathbb{R}}

\renewcommand{\Re}{\operatorname{Re}}
\renewcommand{\Im}{\operatorname{Im}}
\newcommand{\dd}{\mathrm{d}}

\newcommand{\per}{\textnormal{per}}
\newcommand{\loc}{\textnormal{loc}}
 
\newcommand\ep{\varepsilon}

\renewcommand\t{\theta}
\renewcommand\l{\lambda}
\renewcommand\i{\mathrm i}

\renewcommand\d{\partial}

\DeclareMathOperator{\Dt}{\partial_t \hspace{-1.5pt}}
\renewcommand{\div}{\mathrm{div}}

\renewcommand{\L}{\mathcal{L}}

\renewcommand{\P}{\mathcal{P}}

\renewcommand{\l}{\lambda}

\newcommand{\ZZ}{\mathbb{Z}}

\newcommand{\GT}{\mathcal{G}} %Gelfand transform
\newcommand{\RS}{\Gamma_\varepsilon} %Rescaling
\newcommand{\FL}{\operatorname{\mathcal{L}}_\nu} %Fourier--Laplace transform
\newcommand{\FT}{\operatorname{F}} %spatial Fourier transform
\newcommand{\pot}{\textnormal{pot}}
\newcommand{\sol}{\textnormal{sol}}

\title{Fibre homogenisation for time-dependent problems}
\author[1]{Shane Cooper}
\author[2]{Imane Essadeq}
\author[3]{Marcus Waurick}

\affil[1]{\footnotesize Department of Mathematics, University College London, 25 Gordon Street, London, WC1H 0AY.}

\affil[2]{\footnotesize Mathematical Modeling and Numerics (MoNum), EHTP - Hassania School of Public Works, Km 7 El Jadida Road, Casablanca.}

\affil[3]{\footnotesize Department of Mathematics and Computer Science, TU Bergakademie Freiberg, Freiberg, Germany}
\makeatletter
\newcommand{\leqnomode}{\tagsleft@true\let\veqno\@@leqno}
\newcommand{\reqnomode}{\tagsleft@false\let\veqno\@@eqno}
\makeatother

\begin{document}

\maketitle

\begin{abstract}
In this article we provide a method for establishing operator-type error estimates between solutions to rapidly oscillating evolutionary equations and their homogenised counter parts. This method is exemplified by applications to the wave, heat and finally thermoelastic evolutionary systems. 
\end{abstract}
\section{Introduction}
%The homogenisation theory for  differential equations with rapidly oscillating periodic coefficients, is concerned with establishing the asymptotic behaviour of solutions $u_\ep$ to differential equations  in the limit $\ep$ tends to zero. Here, the small parameter $\ep$ is the period of the coefficients. The standard basic result of Homogenisation theory is that $u_0 = \lim\limits_{\ep \rightarrow 0} u_\ep$ exists in a certain function space, and solves some \textit{constant} coefficient differential equation. This type of convergence result is now coined qualitative homogenisation as opposed to a more recent activity of `quantitative' homogenisation in which one attempts to establish rates of convergence of $u_\ep$ to $u_0$ in certain norms.  The key  goal of quantitative homogenisation is to establish that the difference $u_\ep - u_0$  is  order $\ep$ small  in norm,  uniform with respect the the right-hand-side $f$; that is to establish estimates of the form
%\[
%\| u_\ep - u_0 \|_X \lesssim \ep \| f\|_Y
%\]
%for some function spaces $X$ and $Y$, whose naturally choice typically depends on the problem at hand. 

This article is concerned with the quantitative homogenisation of evolutionary equations (in the sense of Picard) with rapidly oscillating  periodic spatial coefficients
\[
\big( \d_t M(\tfrac{\cdot}{\ep}) +N(\tfrac{\cdot}{\ep}) +A\big) U_\ep = F,
\]
where  $M$  and $N$ are bounded and positive coefficient matrices, the small parameter $\ep$ is the spatial period and $A$ is some first order differential operator. The usual, now classical, `qualitative' homogenisation theory states that the solution $U_\ep$ converges (in some function space), as $\ep$ tends to zero, to $U_0$ the solution to an evolutionary equation with constant coefficients
\[
\big( \d_t M_0  +N_0+A\big) U_\ep = F.
\]
where $M_0$ and $N_0$ are the appropriate homogenisation limits of $M(\tfrac{\cdot}{\ep})$ and $N(\tfrac{\cdot}{\ep})$. The qualitative homogenisation  literature needs no introduction as it is now long standing, rich and ubiquitous. That said,  for an example of qualititative homogenisation of evolutionary systems, see \cite{W13}.  Recent developments, in homogenisation theory, have included establishing error estimates between solutions and their homogenised limits in various function spaces; with a priority focused on establishing so-called `operator-type' estimates; that is, estimates uniform with respect to the right-hand-side is certain natural norms. Such activity has come to be coined `quantitative' homogenisation. 

In \cite{BiSu04}, \cite{Zh89} the so-called spectral method was developed to establish quantitative homogenisation results for differential problems with rapidly oscillating periodic coefficients defined in the whole space. This method relies on utilising the Gelfand transform to replace the differential operator with a family of operators, i.e. to represent the operator as a fibre decomposition, and then performing an asymptotic analysis on each fibre operator that is uniform with respect to the fibre decomposition parameter.
In \cite{CoWa19}, the spectral method was extended to the context of `stationary' evolutionary problems that admitted a fibre decomposition via the Gelfand transform. Here, we developed the spectal method further to cover the case of time-dependent evolutionary problems. The key distinction is that the time-dependent operators  admit a two parameter fibre decomposition, a spatial decomposition (due to the Gelfand transform) and a temporal decomposition (due to the Laplace transform); we shall perform an asymptotic analysis of the decomposed operators that is uniform in the both the Gelfand and Laplace transform to arrive at operator-type homogenisation error estimates in temporally-weighted function spaces. We shall apply this method to the quintessential hyperbolic and parabolic systems, i.e. the wave and heat system respectively, then demonstrate the method on a thermo-elastic system of equations, which is an example of a coupled wave-heat system.

\subsection*{Structure of the article}
In Section \ref{s.w}, we study the wave equation 
\begin{equation*}
%	\label{w1}
	\partial_{t}^2 u_\ep - \div ( a(\tfrac{x}{\ep}) \nabla u_\ep ) = f, \quad t\in \RR ,\, x\in \RR^d,
\end{equation*}
which is reformulated as the evolutionary wave system
 \begin{equation*}
% 	\label{ws1}
 	\Bigg(\partial_t \left( \begin{matrix} 1 & 0 \\ 0 & a(\tfrac{\cdot}{\ep})^{-1}
 	\end{matrix} \right) + \left( \begin{matrix} 0 & -\div_x \\ -\nabla_x & 0
 	\end{matrix} \right) \Bigg) U^\ep = \left( \begin{matrix} f \\ 0
 	\end{matrix} \right) \in L_{\nu}^2(\RR;L_2(\RR^d)^{1+d}),
 \end{equation*}
for $U^\ep(t,x) = (\Dt u_\ep(t,x), a(\tfrac{x}{\ep})  \nabla u_\ep(t,x))^\top$. By the spectral method, we establish the following estimate
\[
\Big(\|   U^\ep_1 - U^0_1  \|_{L^2_\nu(\RR;L^2(\RR^d))}^2 + 	\| U^\ep_2 - U^0_2 \|_{L^2_\nu(\RR;H^{-1}(\RR^d;\RR^d))}^2\Big)^{1/2}  \le C \ep \| \d_t^2 f\|_{L^2_\nu(\RR;L^2(\RR^d))},
\]
for some constant $C>0$ independent of $\ep$ and $f$. Here, $U_0$ is the solution to the homogenised evolutionary wave system
\begin{equation*}
	\Bigg(\partial_t \left( \begin{matrix} 1 & 0 \\ 0 & a_0^{-1}
	\end{matrix} \right) + \left( \begin{matrix} 0 & -\div_x \\ -\nabla_x & 0
	\end{matrix} \right) \Bigg) U^0 = \left( \begin{matrix} f \\ 0
	\end{matrix} \right) \in L_{\nu}^2(\RR;L_2(\RR^d)^{1+d}),
\end{equation*}
for $a_0$ the homogenised matrix associated to $a$. As $U^0 = (\Dt u_0, a_0 \nabla u_0)^\top$, for $u_0$ the solution to the homogenised wave equation 
\begin{equation*}
	%	\label{w1}
	\partial_{t}^2 u_0 - \div ( a_0 \nabla u_0 ) = f, \quad t\in \RR ,\, x\in \RR^d,
\end{equation*} these estimates can be rewritten as 
	\begin{align*}
	\| \Dt u_\ep - \Dt u_0 \|_{L^2_\nu(\RR;L^2(\RR^d))} &\le C \ep \| \Dt^2 f\|_{L^2_\nu(\RR;L^2(\RR^d))}, \\
	\| a(\tfrac{\cdot}{\ep}) \nabla u_\ep  - a_0   \nabla  u_0 \|_{L^2_\nu(\RR;H^{-1}(\RR^d;\RR^d))}& \le C \ep \| \Dt^2 f \|_{L^2_\nu(\RR;L^2(\RR^d))}.
\end{align*}
However, a closer look at our approach shows that we additionally establish a corrector estimate
\[
	\| a(\tfrac{\cdot}{\ep}) \nabla u_\ep  - a(\tfrac{\cdot}{\ep} )\nabla ( u_0 + \ep N_0(\tfrac{\cdot}{\ep}) \cdot  \nabla  u_0)  \|_{L^2_\nu(\RR;L^2(\RR^d;\RR^d))} \le C \ep \| \Dt^2 f \|_{L^2_\nu(\RR;L^2(\RR^d))},\\
\]
for some constant $C>0$ independent of $\ep$ and $f$; here, $N_0$ is the classical homogenised corrector. 

In Section \ref{s.h},  we study the heat equation 
\begin{equation*}
	\partial_{t} u_\ep - \div_x ( b(\tfrac{x}{\ep}) \nabla_x u_\ep ) = f, \quad t\in \RR ,\, x\in \RR^d,
\end{equation*}
which is reformulated as the evolutionary heat system
\begin{equation*}
	\Bigg( \left( \begin{matrix} \Dt & 0 \\ 0 & b^{-1}(\tfrac{x}{\ep})
	\end{matrix} \right) + \left( \begin{matrix} 0 & -\div_x \\ -\nabla_x & 0
	\end{matrix} \right) \Bigg) U^\ep = \left( \begin{matrix} f \\ 0
	\end{matrix} \right) \quad t\in \RR,\, x\in \RR^d.
\end{equation*}
for $U^\ep(t,x) = (u_\ep(t,x), b(\tfrac{x}{\ep})  \nabla u_\ep(t,x))^\top$. We establish the following estimate
	\[
	\Big(\|   U^\ep_1 - U^0_1  \|_{L^2_\nu(\RR;L^2(\RR^d))}^2 + 	\| U^\ep_2 - U^0_2 \|_{L^2_\nu(\RR;H^{-1}(\RR^d;\RR^d))}^2\Big)^{1/2}  \le C \ep \| f\|_{L^2_\nu(\RR;L^2(\RR^d))},
	\]
	for $U^0$ the solution to the homogenised evolutionary heat system
	\begin{equation*}
		\Bigg(\left( \begin{matrix} \partial_{t} & 0 \\ 0 & b_0^{-1}
		\end{matrix} \right) + \left( \begin{matrix} 0 & -\div_x \\ -\nabla_x & 0
		\end{matrix} \right) \Bigg) U^0 = \left( \begin{matrix} f \\ 0
		\end{matrix} \right) \quad t\in \RR,\, x\in \RR^d,
	\end{equation*}
	for $b_0$ the homogenised matrix associated to $b$. 	Again, as $U^0 = (u_0, b_0 \nabla u_0)^\top$, these estimates can be rewrriten as 
	\begin{align*}
		\| \Dt^{1/2} u_\ep - \Dt^{1/2}  u_0 \|_{L^2_\nu(\RR;L^2(\RR^d))} \le C \ep \| f\|_{L^2_\nu(\RR;L^2(\RR^d))}, \\
	\| b(\tfrac{\cdot}{\ep}) \nabla u_\ep  - b_0   \nabla  u_0 \|_{L^2_\nu(\RR;H^{-1}(\RR^d;\RR^d))} \le C \ep \| f \|_{L^2_\nu(\RR;L^2(\RR^d))}, 
\end{align*}
for $u_0$ the solution to the homogenised heat equation 
\begin{equation*}
	%	\label{w1}
	\partial_{t} u_0 - \div ( b_0 \nabla u_0 ) = f, \quad t\in \RR ,\, x\in \RR^d.
\end{equation*}
 As in Section 2, we additionally have the corrector estimate
 \[
 	\| b(\tfrac{\cdot}{\ep}) \nabla u_\ep  - b(\tfrac{\cdot}{\ep} )\nabla ( u_0 + \ep N_0(\tfrac{\cdot}{\ep}) \cdot  \nabla  u_0)  \|_{L^2_\nu(\RR;L^2(\RR^d;\RR^d))} \le C \ep \|  f \|_{L^2_\nu(\RR;L^2(\RR^d))}.
 	\]
 Upon comparing these estimates with the wave equation error estimates, we see an improvement of the upper bound in the sense that one has a rougher norm for $f$ with respect to the time variable. This is due to the  existence of maximal regularity estimates  for the evolutionary heat equation (cf. Lemma \ref{p.reg}) which are utilised when quantitatively justifying the fibre homogenisation approach. 

Finally, in Section \ref{s.te}, we study a coupled wave-heat type equation; namely the thermo-eslastic system of equations
\begin{equation*}
	\left\{ \begin{aligned}
		\partial_{t}^2 u_\ep - \div_x \big( a(\tfrac{x}{\ep}) \nabla_x u_\ep \big)+ \gamma(\tfrac{x}{\ep}) v_\ep = f, \qquad t\in \RR, \, x \in \RR^d ,\\
		\d_t v_\ep - \div_x\big( b(\tfrac{x}{\ep}) \nabla_x v_\ep \big) - \overline{\gamma}(\tfrac{x}{\ep}) \d_t u_\ep = g,\qquad t\in \RR, \, x \in \RR^d.
	\end{aligned} \right.
\end{equation*}
Here, we rewrite the thermo-elastic system as the following evolutionary system
\begin{equation}\label{EvoP}
	(\partial_t M_\ep+N_\ep+A)U^\ep = F,  
\end{equation} 
where $U^\ep=(\partial_t u_\ep, a(\tfrac{x}{\ep})\nabla_x u_\ep, v_\ep,b(\tfrac{x}{\ep})\nabla_x v_\ep)^\top$, $F = (f,0,g,0)^\top$,
\begin{multline*}
	M_\ep= \begin{pmatrix}
		1 & 0 & 0 & 0 \\ 0 & a^{-1}(\tfrac{x}{\ep}) & 0 & 0 \\ 0 & 0 & 1 & 0 \\ 0 & 0 & 0 &0 
	\end{pmatrix}, \  N_\ep = \begin{pmatrix}
		0 & 0 & \gamma(\tfrac{x}{\ep})  & 0 \\ 0 & 0 & 0 & 0 \\ -\overline{\gamma}(\tfrac{x}{\ep}) & 0 & 0 & 0 \\ 0 & 0 & 0 &b^{-1}(\tfrac{x}{\ep})
	\end{pmatrix} \text{ and} \\  A = \begin{pmatrix}
		0 & -\div_x  & 0 &0 \\  -\nabla_x  & 0 &  0 &  0 \\ 0 & 0 & 0 & -\div_x \\ 0 & 0 & -\nabla_x &0 
	\end{pmatrix}.
\end{multline*}
Utilising the results of Section \ref{s.w} and \ref{s.h}, we  establish evolutionary system homogenisation estimates and corrector estimates, which we combine as follows: 
	\begin{align*}
	&\notag	\| \Dt u_\ep - \Dt u_0 \|_{L^2_\nu(\RR;L^2(\RR^d))} + 	\| \Dt^{1/2} v_\ep - \Dt^{1/2} v_0 \|_{L^2_\nu(\RR;L^2(\RR^d))} 	\\&+	\| a(\tfrac{\cdot}{\ep}) \nabla u_\ep  - a(\tfrac{\cdot}{\ep} )\nabla ( u_0 + \ep N_0(\tfrac{\cdot}{\ep}) \cdot  \nabla  u_0)  \|_{L^2_\nu(\RR;L^2(\RR^d;\RR^d))} \nonumber \\
	&	+ 	\| b(\tfrac{\cdot}{\ep}) \nabla v_\ep  -b(\tfrac{\cdot}{\ep} )\nabla ( v_0 + \ep N_0(\tfrac{\cdot}{\ep}) \cdot  \nabla  v_0)  \|_{L^2_\nu(\RR;L^2(\RR^d;\RR^d))} 
	+	\| a(\tfrac{\cdot}{\ep}) \nabla u_\ep  - a_0   \nabla  u_0 \|_{L^2_\nu(\RR;H^{-1}(\RR^d;\RR^d))}  \nonumber\\
	&+ 	\| b(\tfrac{\cdot}{\ep}) \nabla v_\ep  - b_0   \nabla  v_0 \|_{L^2_\nu(\RR;H^{-1}(\RR^d;\RR^d))}  \le C \ep \big( \| \Dt^2 f \|_{L^2_\nu(\RR;L^2(\RR^d))} +  \| \Dt  g\|_{L^2_\nu(\RR;L^2(\RR^d))} \big), 
\end{align*}
for some constant $C>0$ independent of $\ep$ and $f$, and $u_0$, $v_0$ solving the coupled homogenised thermo-elastic system
\begin{equation*}
	\left\{ \begin{aligned}
		\partial_{t}^2 u_0 - \div_x a_0 \nabla_x u_0 + \langle \gamma \rangle v_0 = f, \qquad t\in \RR, \, x \in \RR^d ,\\
		\d_t v_0 - \div_x b_0 \nabla_x v_0  - \overline{\langle \gamma \rangle}\Dt u_0 = g,\qquad t\in \RR, \, x \in \RR^d,
	\end{aligned} \right.
\end{equation*}
for $\langle h \rangle$ denoting the integral of a periodic function $h$ over its period.  
\section*{Notation}
For a given Hilbert space $H$ with inner product $\langle \cdot, \cdot \rangle_H$ and norm $\| \cdot \|_H : = \langle \cdot,\cdot\rangle^{1/2}$, we denote $L(H)$ to be the space of bounded linear operators in $H$,  
\[
L^2_{\nu}(\mathbb{R};H)\coloneqq \{ f\in L^2_{\loc}(\mathbb{R};H); \int_{\mathbb{R}} \|f(t)\|_H^2 \exp(-2\nu t)\dd t<\infty\}, \quad \nu>0,
\]
and
 $H_\nu^1(\mathbb{R};H)$ is the space of $H$-valued, weakly differentiable $L_{2,\nu}(\mathbb{R};H)$-functions with distributional derivative representable as an $L^2_{\nu}(\mathbb{R};H)$-function. $L^2_\nu(\RR;H)$ and $H^1_\nu(\RR;H)$ are Hilbert spaces with the standard inner products. For a given function space $X$, $X_{\per}$ denotes the space of periodic functions that locally belong to $X$, for example: 
 $W^{1,p}_{\per}(\square) : = \{ f \in W^{1,p}_{\loc}(\RR^d) \, : \, \text{$f$ is periodic with periodicity cell $\square$}  \}$.
 
 Throughout $\lesssim$ denotes less than or equal to up to some constant that is independent of the underlying parameters.

\section{Wave equation}\label{s.w}
Consider the heterogeneous wave equation with rapidly oscillating coefficients:
\begin{equation}\label{w1}
\partial_{t}^2 u_\ep - \div_x ( a(\tfrac{x}{\ep}) \nabla_x u_\ep ) = f, \quad t\in \RR ,\, x\in \RR^d,
\end{equation}
where $\ep>0$ is a fixed parameter (the small period), $a$ is a periodic function with periodicity cell $\square : = [-\tfrac{1}{2}, \tfrac{1}{2})^d$, $f\in L_\nu^2(\RR;L_2(\RR^d))$ is given for some $\nu>0$ and $u_\ep$ is the unknown; $\div$ and $\nabla$ are the usual divergence and gradient differential operators (with respect to the spatial variable). 

We shall determine the leading-order asymptotics of $u^\ep$  by using an appropriate modification, given in \cite{CoWa19}, of the so-called spectral method (c.f. \cite{Zh89,BiSu04}) to `evolutionary equations' of Picard, see \cite{P09} for the seminal research article and \cite{STW22} for an (easily) accessible monograph on the subject matter.

We begin by following \cite[pp 92]{STW22} and noting that if $u_\ep$ satisfies \eqref{w1} then the vector $U^\ep(t,x) = (\Dt u_\ep(t,x), a(\tfrac{x}{\ep})  \nabla u_\ep(t,x))^\top$ solves the evolutionary wave system   
\begin{equation}
	\label{ws1}
	\Bigg(\partial_t \left( \begin{matrix} 1 & 0 \\ 0 & a(\tfrac{\cdot}{\ep})^{-1}
	\end{matrix} \right) + \left( \begin{matrix} 0 & -\div_x \\ -\nabla_x & 0
	\end{matrix} \right) \Bigg) U^\ep = \left( \begin{matrix} f \\ 0
	\end{matrix} \right) \in L_{\nu}^2(\RR;L_2(\RR^d)^{1+d}).
	\end{equation}
Conversely, if $U^\ep$ solves \eqref{ws1} then $u_\ep : =\Dt^{-1} U^\ep_1 \in H_\nu^1(\RR;L^2(\RR^d))$ solves the wave equation \eqref{w1} in a distributional sense.
%That is to say that the problems \eqref{w1} and \eqref{ws1} are equivalent.

Considering solutions of \eqref{ws1}, in the Bochner space $L^2_\nu(\RR; L^2(\RR^d))$, for $\nu>0$,  with time $t$ being the Bochner variable, we apply Picard's theorem as in \cite[Theorem 6.2.6]{STW22} and obtain the following result. 

\begin{proposition}\label{prop:PTwave} Let $\nu>0$, $f\in L_\nu^2(\RR;L^2(\RR^d))$. Assume $a\in L^\infty(\mathbb{\square};L(\mathbb{C}^d))$ with $a=a^*$ pointwise a.e.~and that there exists $\kappa>0$ such that
\[
   a(y) \xi \cdot \overline{\xi} \ge \kappa |\xi|^2, \quad( \xi \in \RR^d,  y \in \square)
\]
Then there exists a unique solution $U^\ep \in L_\nu^2(\RR;L^2(\RR^{d}))^{1+d}$ of \eqref{ws1}. If, in addition, $f\in H_\nu^1(\RR;L^2(\RR^{d}))$, then
\[
U^\ep \in H_\nu^1(\RR;L^2(\RR^{d}))^{d+1} \cap L^2_\nu(\RR; {\rm dom}\,  \left( \begin{matrix} 0 & -\div \\ -\nabla & 0
\end{matrix} \right)).
\]
\end{proposition}
\begin{remark}\label{rem:coerc} It is clear that (cf. \cite[Proposition 6.2.3 (b)]{STW22}), $a\in L^\infty(\mathbb{\square};L(\mathbb{C}^d))$ with 
\[
  \Re a(y) \xi \cdot \overline{\xi} \ge \kappa |\xi|^2, \quad (\xi \in \RR^d, y \in \square)
\]
for some $\kappa>0$ is equivalent to
\begin{equation}\label{coeffass}
	a^{-1}\in L^\infty(\mathbb{\square};L(\mathbb{C}^d))\text{ and }\exists\kappa >0\forall\xi \in \RR^d, \,  y \in \square\colon \Re a^{-1}(y) \xi \cdot \overline{\xi }\ge \kappa |\xi|^2.
\end{equation} 
\end{remark}
%\kappa |\xi|^2 \le {\rm Re}\, a(y) \xi \cdot \overline{\xi} \le \kappa^{-1} |\xi|^2, \quad \forall\, \xi \in \RR^d, \, \forall\, y \in \square, 

To apply the spectral method for determining the quantitative asymptotics of $U^\ep$ we apply a series of invertible transformations.
% In order to do that in a rigorous manner, we assume throughout 
%\[
%   f = f_0\cdot f_1 \in C_c^\infty(\RR)\cdot C_c^\infty(\RR^d)\subseteq H_\nu^1(\RR;L_2(\RR^d))\subseteq L_{2,\nu}(\RR;L_2(\RR^d)).
%\]
%By density, all the estimate provided will carry through to $f\in L_{2,\nu}(\RR;L_2(\RR^d))$ by appropriate continuity arguments. The particular derivations of the estimate are more easily derived by pointwise arguments, hence the assumption on $f$.
%In any case, 
First, an application of the (bilateral) Laplace transform
\[\FL: L^2_{\nu}(\mathbb{R};L^2(\RR^d;\CC\times \CC^{d}))\to L^2(\mathbb{R};L^2(\RR^d;\CC\times \CC^{d}))\] given as the (unitary) continuous extension of the operator defined by
\[
( \FL \phi)(k)\coloneqq 
\frac{1}{\sqrt{2\pi}}\int_{\mathbb{R}} \phi(t) e^{-\lambda t}\dd t,\qquad\phi\in C_c(\mathbb{R};L^2(\RR^d;\CC\times \CC^{d})), \lambda = \nu + \i k,  k\in \mathbb{R},
\]
and the fact (cf. \cite[Theorem 5.2.3]{STW22}) that
\begin{equation}
	\label{tp1}
	(\FL \Dt g)(k)= \l (\FL g)(k)\quad(g \in H_\nu^1(\RR;L^2(\RR^d)^{d+1}), \lambda = \nu+\i k, k \in\RR)
\end{equation}	
determines that $U^\ep_{\lambda} := \L_\nu U^\ep(k)$, $\l=\nu+\i k$,  solves 
\begin{equation*}
	\Bigg( \l  \left( \begin{matrix} 1 & 0 \\ 0 & a(\tfrac{\cdot}{\varepsilon})^{-1}
	\end{matrix} \right) + \left( \begin{matrix} 0 & -\div \\ -\nabla & 0
	\end{matrix} \right) \Bigg) U^\ep_{\l}= \left( \begin{matrix} \L _\nu f(k) \\ 0
	\end{matrix} \right),\qquad \l =\nu+ \i k, k \in \RR, \ep >0.
	%	p})+N(\tfrac{x}{\ep})+A)U^\ep_{\l} = \L F(\l), \qquad \l =\nu+ \i k, k \in \RR, \ep >0.
	\end{equation*}

	Then, upon an application of the (spatial) re-scaling 
	\[
	\RS : L^2(\RR^d;\CC \times \CC^{d}) \rightarrow L^2(\RR^d;\CC\times \CC^{d}),
	\] given by the action $\RS u(x) := u(\ep x)$, and using the fact that
	\begin{equation}
		\label{tp3} \nabla \RS = \ep \RS \nabla,
	\end{equation}
	 we find that $W^\ep_{\l} : = \RS U^\ep_{\l}$ solves 
	\begin{equation*}
\Bigg( \l  \left( \begin{matrix} 1 & 0 \\ 0 & a^{-1}
\end{matrix} \right) + \ep^{-1}\left( \begin{matrix} 0 & -\div \\ -\nabla & 0
\end{matrix} \right) \Bigg) W^\ep_{\l}= \left( \begin{matrix} \RS \L_\nu f(k) \\ 0
\end{matrix} \right),\qquad \l =\nu+ \i k, k \in \RR, \ep >0.
%	p})+N(\tfrac{x}{\ep})+A)U^\ep_{\l} = \L F(\l), \qquad \l =\nu+ \i k, k \in \RR, \ep >0.
\end{equation*} 

Finally, upon an  application of the Gelfand transform 
\[\GT : L^2(\RR^d;\CC \times \CC^d)\rightarrow L^2(\square^* \times \square ; \CC \times \CC^d),\quad \square^* :=[-\pi,\pi)^d,\] given as the continuous extension of 
\[
(\GT \phi)(\t,y)\coloneqq (2\pi)^{-d/2} \sum_{z \in \ZZ^d } \phi(y+z) e^{-\i \t \cdot(y + z)}, \quad \phi \in C_c^\infty(\RR^d;\CC \times \CC^d),
\] with its inverse given by
\[
  (\GT^{-1} \Phi)(x) =\frac{1}{(2\pi)^{d/2}} \int_{\square^*} \Phi(\t,x)e^{\i \t\cdot x}\dd \t\quad(x\in \RR^d),
\]
it follows from the properties
\begin{align}
	\label{tp2} &\GT \nabla  = (\nabla_y + \i \t) \GT ;\\
	\label{tp4} &\GT p = p \GT, \ \text{where $p$ is multiplication by a periodic function}. 
\end{align}
 that $U^\ep_{\l,\t} : = (\GT W^\ep_{\l})(\t,\cdot)$
% \in H_{\per} : = \big( H^1_{\per}(\square) \times H_{\rm div,per}(\square) \big)^2$ 
is the $\square$-periodic solution to
\begin{equation}\label{waveP1}
\Bigg( \l  \left( \begin{matrix} 1 & 0 \\ 0 & a^{-1}
\end{matrix} \right) + \ep^{-1} \left( \begin{matrix} 0 & -(\div + \i \t \cdot) \\ -(\nabla + \i \t) & 0
\end{matrix} \right) \Bigg) U^\ep_{\l,\t}= \left( \begin{matrix} F \\ 0
\end{matrix} \right),
\end{equation}
for $F : =  (\GT  \RS \L_\nu f(k))(\t,\cdot) \in L^2_{\per}(\square; \CC \times \CC^{d})$.\footnote{Whilst $F$ formally depends on $\ep$, $\l$ and $\t$, the right-hand-side will be fixed throughout and for this reason we drop explicit reference to this dependence.} Here $\i \t $ and $\i \t \cdot$ are understood as (constant) multiplication operators. 

To identify the $\ep$-leading-order asymptotics for $U^\ep_{\l,\t}$, with error estimates that are uniform in $\t,\l$ and $f$, we build upon the method of \cite{CoWa19}, in Section \ref{s.wass}. For this, we introduce $N_0$, the so-called \textit{homogenised corrector}, whose $j$-th component is the unique solution $N^j_0 \in H^1_{\per,0}(\square): =\{ p \in H^1_{\per}(\square) \, | \, {\langle p,1\rangle_{L^2(\square)}=0 \}}$ to  \begin{equation}\label{thetacell}
	\langle a\nabla N^j_0 ,\nabla \phi \rangle_{L^2(\square)} = -\langle a e_j ,\nabla \phi \rangle_{L^2(\square)}, \quad ( \phi \in H^1_{\per}(\square)),
\end{equation}	
	for $e_j$ the $j$-th Euclidean basis vector. Moreover, let the matrix $a_0$ be the \textit{homogenised matrix corresponding to} $a$, that is, $a_0$ is given by 
	\[(a_0)_{jk}  = \langle a( e_j + \nabla N_0^j), e_k \rangle_{L^2(\square)}.\] The main result for the wave equation reads as follows.

\begin{theorem}\label{WaveMainthm}
	Let $\ep>0, \t \in \square^* , \l=\nu+ik, k\in \RR$,  and $F\in L^2_{\per}(\square; \CC \times \CC^{d})$. Consider the solution $U^\ep_{\l,\t}$ to \eqref{waveP1}  and 
	\begin{equation}\label{Wrep}
		W_{\l,\xi} =   c_0(\l,\xi) \left( \begin{matrix}
			\l  \\  a(I+ \nabla N_0 )	\i \xi
		\end{matrix} \right), \quad c_0(\l,\xi) := (\l^2 + a_0 \xi \cdot \xi )^{-1} \langle F,1 \rangle_{L^2(\square)}, \quad \xi \in \RR^d.
	\end{equation}
	%\end{equation} is an approximation to $V^\epsilon_{\lambda,\theta}$.
	%	$V^\ep_{\l,\t} \in E_0  = \{  c +  (
	%	0, s)^\top \, | \, c \in \CC \times \CC^{d}, {\rm div} s = 0  \in (H^1_{\per}(\square))^* \text{and } \langle s,1\rangle_{L^2(\square)}=0\}$ the solution to 
	%	\begin{equation}\label{wavePhom} 
		%		\Bigg(\lambda P_\theta \left( \begin{matrix} 1 & 0 \\ 0 & a^{-1}
			%		\end{matrix} \right)  +  \ep^{-1}\left( \begin{matrix} 0 & -(\div + \i \t \cdot) \\ -(\nabla + \i \t) & 0
			%		\end{matrix} \right) \Bigg)V^\ep_{\l,\t}  =P_\theta  \left( \begin{matrix} F\\ 0
			%		\end{matrix} \right),
		%	\end{equation} 
	%where $P_\theta : L^2(\square;\CC\times \CC^d) \rightarrow E_\theta$ is the orthogonal projection. 	
	
	Then, there exists $C>0$ uniformly in $\ep, \l, \t$ and $F$ such that
	\begin{equation}
		\label{W.errest1}
		\| U^\ep_{\l,\t} - W_{\l,\t/\ep}  \|_{L^2(\square; \CC^{d+1})} \le C \ep | \l |^2 \| F\|_{L^2(\square)}.
	\end{equation}
\end{theorem}

%\begin{remark}Even though it does follow from the previous theorem, for later purposes, it makes sense to realise that
%\[
%   k\mapsto W_{{(\i k+\nu)},\xi} \in L_{2}(\RR;L_2(\square)^{d+1})
%\]independently of Theorem \ref{WaveMainthm}.
%For this, consider the equation satisfied by $W_{\lambda,\xi}$ and standard real-part estimates yield the desired properties.
%\end{remark}
This theorem allows us to invert the transformations, applied above, to find  error estimates between the solutions to inhomogeneous and homogenised equations.   Indeed, by the unitary properties of  $\FL$, $\GT$, the fact $\Gamma^\ep$ is a change of variables, \eqref{tp1} and  \eqref{W.errest1}  implies
\begin{equation}\label{6.7.23e1}
	\| U^\ep - W^{\ep}  \|_{L^2_\nu(\RR;L^2(\square;\CC\times \CC^{d}))} \le C \ep \| \Dt^2 f\|_{L^2_\nu(\RR;L^2(\RR^d))}
\end{equation}
for the same $C$ as in \eqref{W.errest1} and $W^\ep : = \FL^{-1} \RS^{-1} \GT^{-1} W_{\l, \t / \ep}$.  Now by \eqref{tp1}, it is clear that 
\[
W^\ep_1 = \FL^{-1} \RS^{-1} \GT^{-1} \big(\lambda c_0(\lambda,\tfrac{\theta}{\ep}) \big) = \Dt \mathbf{W} \qquad \text{where $\mathbf{W} : = \FL^{-1}\RS^{-1} \GT^{-1}  c_0$}.
\]
We emphasise that $\RS$ and $\GT$ are transformations with respect to the \emph{spatial} variables, whereas $\FL$ acts on the \emph{temporal} variable and $\FL^{-1}$ on the \emph{frequency} variable $\lambda = \nu + \i k$, $k\in \RR$.
Let us determine 
\[W^\ep_2 = \FL^{-1} \RS^{-1} \GT^{-1}\Big(a(y)(I+ \nabla N_0(y)) \i \tfrac{\t}{\ep} c_{0}(\lambda, \tfrac{\theta}{\ep})\Big) \] in terms of $\mathbf{W}$; by \eqref{tp4}, \eqref{tp2} and the fact that $\nabla_y c_{0}=0$ one has  
 \[
\GT^{-1}\Big( a(y)(I+ \nabla N_0(y)) \i \tfrac{\t}{\ep} c_{0}(\lambda,\tfrac{\t}{\ep})\Big) = a(I + \nabla N_0) \ep^{-1}\nabla \GT^{-1} c_{0};
\] 
and then, by \eqref{tp3}, we get
\[
\Gamma^{-1}_\ep(a(I+ \nabla N_0) \ep^{-1}  \nabla \GT^{-1} c_{0} = a(\tfrac{\cdot}{\ep})(I + \nabla N_0(\tfrac{\cdot}{\ep})) \nabla (\RS^{-1} \GT^{-1} c_{0}).
\]
That is,  
\[
W^\ep_2 = a(\tfrac{x}{\ep})(I + \nabla N_0(\tfrac{x}{\ep})) \nabla \mathbf{W}.
\]
Therefore, $w^\ep = (\Dt \mathbf{W},  a(\tfrac{x}{\ep})(I + \nabla N_0(\tfrac{x}{\ep})) \nabla \mathbf{W})^\top$ and in the spirit of homogenisation theory, it is instructive to determine the equation that $\mathbf{W}$ solves.  For this, we shall use the following known (and readily verifiable) relations between the  Gelfand transform, $\GT$, and Fourier transform $\FT$:
%\begin{proposition} (a) Let $h\in L^2(\RR^d)$. Then, on $\square^*$, 
\begin{equation}\label{GvsF}
	\FT h = \langle \GT h, 1 \rangle_{L^2(\square)},
	\end{equation}
and
	\begin{equation}\label{GvsF2}
\text{for $g\in L^2(\square^*)$, i.e. independent of $y$, then	$\GT^{-1} g = \FT^{-1}( \chi g)$, }
	\end{equation} 
	where $\chi$ is the characteristic function of $\square^*$.
%\end{proposition} 
%\begin{proof}
%(a) Due to continuity of the considered operators, it suffices to prove the statement for $h$ continuous with compact support. For this we compute for $\t\in \square^*$
%\begin{align*}
%    \langle \GT h, 1 \rangle_{L^2(\square)}(\t) & = \langle (\GT h)(\t,\cdot), 1\rangle_{L^2(\square)} \\
%    & = \int_{\square}(2\pi)^{-d/2} \sum_{z \in \ZZ^d } h(y+z) e^{-\i \t \cdot(y + z)} \dd y \\
%     & = \sum_{z \in \ZZ^d }\int_{\square}(2\pi)^{-d/2} h(y+z) e^{-\i \t \cdot(y + z)} \dd y \\
%     & =(2\pi)^{-d/2}\int_{\RR^d} h(y)e^{-\i \t \cdot y} \dd y = \FT h(\t),
%\end{align*}
%which proves the claim in (a).
%
%(b) Again, by density, it suffices to take $g$ to be a continuous function. Then we compute for $x\in \RR^d$
%\begin{align*}
%  (\GT^{-1}g) (x) &= (2\pi)^{-d/2}\int_{\square^*} g(\t)e^{\i \t x}\dd \t \\
%    & = (2\pi)^{-d/2}\int_{\RR^d} \chi(\t)g(\t)e^{\i \t x}\dd \t  = (\FT^{-1} \chi g)(x).\qedhere
%\end{align*}
%\end{proof}
With these properties, we revisit  \eqref{Wrep} to establish the following:

%\begin{proposition}\label{prop:c0}
As $c_{0}$ satisfies
\[
\l^2 c_{0}(\lambda,\tfrac{\t}{\ep}) + a_0 \tfrac{\t}{\ep} \cdot  \tfrac{\t}{\ep} c_{0}(\l,\tfrac{\t}{\ep}) = \int_\square F(\t,y) \, \mathrm{d}y, \quad F = \GT \RS \FL f, 
\]
then $\mathbf{W}= \FL^{-1}\RS^{-1} \GT^{-1}  c_{0}$ satisfies
\begin{equation}\label{W1eqn}
	\Dt^2 \mathbf{W}  - \div\,  a_0 \nabla \mathbf{W} = \mathcal{P}_\ep f, \quad \mathcal{P}_\ep : = \FT^{-1} \chi \RS^{-1} \FT \RS
\end{equation}
%\end{proposition}
%\begin{proof}
Indeed, the left-hand-side used \eqref{tp1}, \eqref{tp2} with $\nabla_y c_0=0$ and \eqref{tp3}; the right-hand-side used \eqref{GvsF} and then \eqref{GvsF2}. 
%\end{proof}
\begin{remark}
 Direct calculation shows  $\mathcal{P}_\ep =  \FT^{-1} (\RS\chi) \FT$; that is $\mathcal{P}_\ep$ is a self-adjoint smoothing operator, given by truncating a function to the region $\ep^{-1} \square^* = [\ep^{-1} \pi, \ep^{-1} \pi)^d$ in Fourier space, that also appears in error estimates in homogenisation theory in the work \cite{BiSu04}. 
 \end{remark}

As  \eqref{W1eqn} is a constant coefficient equation then 
\begin{equation*}
	\mathbf{W} = \mathcal{P}_\ep u_0, 
\end{equation*}
where $u_0$ solves
\begin{equation}\label{wh1}
\partial_t^2\, u_0 - \div\, a_0 \nabla u_0 = f.
\end{equation}
That is $W^\ep = (\Dt \mathcal{P_\ep} u_0, a(\tfrac{x}{\ep})(I + \nabla N_0(\tfrac{x}{\ep})) \nabla \mathcal{P}_\ep u_0)^\top$ and \eqref{6.7.23e1} reads
\begin{multline}\label{7.7.23e1}
\Big(	\| U^\ep_1 - \Dt \mathcal{P_\ep} u_0  \|_{L^2_\nu(\RR;L^2(\square))}^2 + 	\| U^\ep_2  - a(\tfrac{x}{\ep})(I + \nabla N_0(\tfrac{x}{\ep})) \nabla \mathcal{P}_\ep u_0 \|_{L^2_\nu(\RR;L^2(\square;\CC^d))}^2  \Big)^{1/2} \\  \le  C \ep \| \Dt^2 f\|_{L^2_\nu(\RR;L^2(\RR^d))}.
\end{multline}

It is desirable to remove the smoothing operator $\mathcal{P}_\ep$  in the inequality \eqref{7.7.23e1} and arrive at estimates between  $U^\ep=(\Dt u_\ep, a(\tfrac{x}{\ep}) \nabla u_\ep)^\top$ and $U^0 = (\Dt u_0, a_0 \nabla u_0)^\top$ the solution to the homogenised evolutionary wave equation 
\begin{equation}
	\label{hws1}
	\Bigg(\partial_t \left( \begin{matrix} 1 & 0 \\ 0 & a_0^{-1}
	\end{matrix} \right) + \left( \begin{matrix} 0 & -\div_x \\ -\nabla_x & 0
	\end{matrix} \right) \Bigg) U^0 = \left( \begin{matrix} f \\ 0
	\end{matrix} \right) \quad t\in \RR,\, x\in \RR^d.
\end{equation}
For this, we rely on the properties of $\mathcal{P}_\ep$, the properties of the corrector $N_0$ and the following standard regularity results for $u_0$:
\begin{equation}\label{u0reg}
	\|u_0\|_{L^2_\nu(\RR; H^2(\RR^d))}+ \| \Dt^2 u_0 \|_{L^2_\nu(\RR; L^2(\RR^d))}  + \| \nabla_x \d_t u_0 \|_{L^2_\nu(\RR;H^1(\square;\CC^d)} \lesssim \| \d_t^2 f \|_{L^2_\nu(\RR;L^2(\square;\CC^d)}
\end{equation}
   Indeed, the following result, proved in Section \ref{ss1.2proof}, holds:
\begin{theorem}\label{whomthm}
	Fix $\ep>0$ and $f \in H^2_\nu(\RR; L^2(\RR^d))$, and suppose \eqref{coeffass} holds. Let $u_\ep$ solve \eqref{w1}, $U^\ep = (\Dt u_\ep, a(\tfrac{x}{\ep} )\nabla u_\ep)^\top $ solve \eqref{ws1}, $u_0$ solve \eqref{wh1} and $U^0 = (\Dt u_0, a_0 \nabla u_0)^\top$ solve \eqref{hws1}.
	Then, the following inequalities hold:
	\begin{align}
		\| \Dt u_\ep - \Dt u_0 \|_{L^2_\nu(\RR;L^2(\RR^d))} &\le C \ep \| \Dt^2 f\|_{L^2_\nu(\RR;L^2(\RR^d))}, \label{7.6.23m1}\\
		\| a(\tfrac{\cdot}{\ep}) \nabla u_\ep  - a(\tfrac{\cdot}{\ep} )\nabla ( u_0 + \ep N_0(\tfrac{\cdot}{\ep}) \cdot  \nabla  u_0)  \|_{L^2_\nu(\RR;L^2(\RR^d;\RR^d))} &\le C \ep \| \Dt^2 f \|_{L^2_\nu(\RR;L^2(\RR^d))},\label{7.6.23m2} \\
		\| a(\tfrac{\cdot}{\ep}) \nabla u_\ep  - a_0   \nabla  u_0 \|_{L^2_\nu(\RR;H^{-1}(\RR^d;\RR^d))}& \le C \ep \| \Dt^2 f \|_{L^2_\nu(\RR;L^2(\RR^d))}\label{7.6.23m3}, 
	\end{align}
	for some $C>0$ independent of $\ep$ and $f$.
\end{theorem}
Notice that, since $\Dt$ is boundedly invertible in $L^2_\nu(\RR^d; L^2(\RR^d))$ for $\nu>0$, \eqref{7.6.23m1} implies an error estimate for the homogenisation of the wave equation. In turn, the uniform positivity of $a(\tfrac{\cdot}{\ep})$ along with \eqref{7.6.23m2} imply the so-called corrector estimate. The estimates \eqref{7.6.23m1}  and \eqref{7.6.23m3} provide an error estimate for the homogenisation of the evolutionary wave system. That is, the following inequalities hold:
%These results are gathered in the following statement.
% the proof of which we just have commented on in a sufficient detail.
\begin{corollary}
Under the assumptions in Theorem \ref{whomthm}, we have
\begin{align*}
\|   u_\ep - u_0  \|_{L^2_\nu(\RR;L^2(\RR^d))} & \le C\ep \| \Dt f\|_{L^2_\nu(\RR;L^2(\RR^d))}\\
\|   u_\ep - (u_0 + \ep N_0(\tfrac{\cdot}{\ep}) \cdot \nabla u_0) \|_{L^2_\nu(\RR;H^1(\RR^d))} &\le C\ep \| \Dt^2 f\|_{L^2_\nu(\RR;L^2(\RR^d))}\\
\Big(\|   U^\ep_1 - U^0_1  \|_{L^2_\nu(\RR;L^2(\RR^d))}^2 + 	\| U^\ep_2 - U^0_2 \|_{L^2_\nu(\RR;H^{-1}(\RR^d;\RR^d))}^2\Big)^{1/2} &\le C \ep \| \Dt^2 f\|_{L^2_\nu(\RR;L^2(\RR^d))}.
\end{align*}
\end{corollary}

\subsection{Proof of Theorem \ref{WaveMainthm}}\label{s.wass}
The main step is to  identify a reducing subspace $E_\t$ of  (the closed linear skew-self-adjoint operator in $L^2(\square;\CC \times \CC^d)$)  $A_\t : =  \left( \begin{matrix} 0 & -(\div + \i \t \cdot) \\ -(\nabla + \i \t) & 0
\end{matrix} \right)$ that has the additional property that $A_\t$ restricted to the complement of $E_\t$ is uniformly invertible. In this example, it is easy to determine such an $E_\t$. Indeed,
one readily has the orthogonal decompositions:
% (whose proof is routine): 
\begin{equation}\label{maindecomp}
	L^2(\square) =\CC \oplus \{ f \in L^2(\square) \, | \, \langle f,1\rangle_{L^2(\square)}=0 \}\qquad \text{and} \qquad L^2(\square;\CC^d) =L^2_{\pot,\t}(\square) \oplus  L^2_{\sol, \t}(\square)
	% \{ c + (\nabla + \i \t) p \, | \, c \in \CC^d, p  \in H^1_{per,0}(\square)\} \oplus  L^2_{\sol, \t}(\square)
\end{equation}
where
% $H^1_{per,0}(\square) : = \{ p \in H^1_{\per}(\square) \, | \, \langle p,1\rangle_{L^2(\square)}=0 \}$ and 
\[
L^2_{\pot, \t}(\square) : =\left\{  \begin{array}{lc}
	\{ (\nabla + \i \t) p  \, | \,  p \in H^1_{\per}(\square) \} & \t \neq 0, \\
	\{ c + \nabla p \, |\,  c \in \CC^d \text{ and } p \in H^1_{\per}(\square), \langle p,1 \rangle_{L^2(\square)}=0 \} & \t = 0,
\end{array} \right.
\]
and
\[
L^2_{\sol, \t}(\square) : = \left\{   \begin{array}{lc}
	\{ s  \in L^2(\square;\CC^d) \, | \, \div\, s +\i \t \cdot s = 0 \text { in } (H^1_{\per}(\square))^* \} & \t \neq 0, \\
	\{ s  \in L^2(\square;\CC^d) \, | \, \div\, s  = 0 \text { in } (H^1_{\per}(\square))^* \text{ and }  \langle s,1\rangle_{L^2(\square)}=0\} & \t = 0.
\end{array} \right.
\]
%The first decomposition in \eqref{maindecomp} is obvious and the second is standard and can be seen, for example, by noting that $\nabla + \i \t : H^1_{\per}(\square) \subset L^2(\square) \rightarrow L^2(\square;\CC^d)$ has closed range (see \eqref{Poinc1} for $\t \neq 0$) with  its adjoint being $\div + \i \t \cdot$. 
Furthermore, it is clear from \eqref{maindecomp}$_1$ that 
\begin{equation}\label{L2potrep}
	L^2_{\pot,\t}(\square) \subseteq  \{  \xi + (\nabla + \i \t)p \, | \, \xi \in \CC^d,  p \in H^1_{\per,0}(\square)  \},
\end{equation}
where
\[
   H^1_{\per,0}(\square)\coloneqq \{p\in H^1_{\per}(\square)\, | \, \langle p,1 \rangle_{L^2(\square)}=0\}.
\]
This identification along with another application of \eqref{maindecomp} gives the following orthogonal decomposition:
%\begin{proposition}\label{p.decomp}
\begin{equation}\label{p.decomp}
	L^2(\square; \CC\times\CC^{d}) = E_\t\, \oplus\, E^\perp_\t, \quad \t \in \square^*,
\end{equation}
where 
\begin{equation}\label{E1}
	\begin{aligned}
		&E_\t : =  \{  c +  (
		0, s)^\top \, | \, c \in \CC \times \CC^{d}, s \in L^2_{\sol,\t}(\square)\}; \\
		&E^\perp_\t : = \{ (p_1, (\nabla + \i \t) p_2)^\top
		\, | \, p_1 \in  L^2(\square), p_2 \in H^1_{\per}(\square), \langle p_i, 1 \rangle_{L^2(\square)}=0 , i\in \{1,2\}\}.
	\end{aligned}
\end{equation}
\begin{remark}\label{remEform}
Notice that for $P_\t : L^2(\square;\CC \times \CC^d) \rightarrow E_\t$ the orthogonal projection, one has $P_\t(u_1,u_2)^\top =(P_1 u_1, P_2 u_2)^\top$ where $P_1: L^2(\square) \rightarrow \CC$ and $P_2 : L^2(\square;\CC^d) \rightarrow \CC^d \oplus L^2_{\sol,\t}(\square)$ are the orthogonal projections. Observe that $P_1$ is independent of $\t$. 
\end{remark}

The following result was proved in \cite[Section 3]{CoWa19}:
%(see appendix for a proof):
\begin{proposition}\label{prop.diagA}
%One has the  orthogonal decompositions:
%%\begin{proposition}\label{p.decomp}
%\begin{equation}\label{p.decomp}
%	L^2(\square; \CC\times\CC^{d}) = E_\t\, \oplus\, E_\t^\perp, \quad \t \in \square^*,
%\end{equation}
%for  
%\begin{equation}\label{E1}
%	\begin{aligned}
%		&E_\t : = \{  c +  (
%		0, s)^\top \, | \, c \in \CC \times \CC^{d}, s \in L^2_{\sol,\t}(\square)\}; \\
%		&E_\t^\perp : = \{ (p_1, (\nabla + \i \t) p_2)^\top
%		\, | \, p_1 \in  L^2(\square), p_2 \in H^1_{\per}(\square), \langle p_i, 1 \rangle_{L^2(\square)}=0 \}.
%	\end{aligned}
%\end{equation}
%Here, \[
%L^2_{\sol, \t}(\square) : = \left\{   \begin{array}{lc}
%	\{ s  \in L^2(\square;\CC^d) \, | \, \div\, s +\i \t \cdot s = 0 \text { in } (H^1_{\per}(\square))^* \} & \t \neq 0, \\[.1em]
%	\{ s  \in L^2(\square;\CC^d) \, | \, \div\, s  = 0 \text { in } (H^1_{\per}(\square))^* \text{ and }  \langle s,1\rangle_{L^2(\square)}=0\} & \t = 0.
%\end{array} \right.
%\]
The closed subspace $E_\t$ reduces $A_\t$; namely, 
\begin{equation*}
P_\t A_\t \subseteq A_\t P_\t.
\end{equation*}
Moreover, $A_\t P_\t^\perp$ is uniformly invertible; i.e., \begin{equation}\label{A2ubi}
\sup_{\t \in \square^*} \| (A_\t P_\t^\perp)^{-1}  \|_{L(E_\t^\perp)} <\infty
\end{equation}

% for $D_\t :=E_\t \cap {\rm dom} A_\t $ and $D_\t^\perp :=E_\t^\perp \cap {\rm dom} A_\t $ one has 
%\begin{equation}
%	\begin{aligned}
%P_\t {\rm dom} T = D_\t, \quad  P_\t^\perp {\rm dom} T = D_\t^\perp,  \quad T D_\t \subseteqeq E_\t,  \quad  T D_\t^\perp \subseteq E_\t^\perp.
%	\end{aligned}
%\end{equation}
\end{proposition}
 Equipped with the above diagonalisation of $A_\t$ we can readily provide our first approximation for  $U^\ep_{\l,\t}$.  
\begin{proposition}\label{prop.W1stest}
Consider $U^\ep_{\l,\t}$ the solution to \eqref{waveP1}  and $V^\ep_{\l,\t}\in E_\t$ the solution to 
	\begin{equation}\label{waveP2}
		\Bigg(\lambda P_\t \left( \begin{matrix} 1 & 0 \\ 0 & a^{-1}
		\end{matrix} \right)  + \ep^{-1}  \left( \begin{matrix} 0 & -(\div + \i \t \cdot) \\ -(\nabla + \i \t) & 0
	\end{matrix} \right) \Bigg)V^\ep_{\l,\t} =P_\t  \left( \begin{matrix} F\\ 0
\end{matrix} \right).
	\end{equation} 

Then, the following estimate holds: there exists $C>0$ independent of $\ep, \l, \t$ and $F$ such that  
	\begin{equation}\label{W.errest1.1}
		\| U^\ep_{\l,\t} - V^\ep_{\l,\t} \|_{L^2(\square; \CC^{d+1})} \le C \ep | \l |^2 \| F\|_{L^2(\square)}.
	\end{equation}
\end{proposition}
\begin{proof}
Let us drop super and subscripts and $\| \cdot \|$ shall denote the appropriate $L^2$-norm. 

Since $A_\t$ and $P_\t$ commute (cf.\ Proposition \ref{prop.diagA}), then it is known (see for example \cite[Chapter 3, Section 6]{BiSol}) that $A_\t = A_\t^{(1)} + A_\t^{(2)}$, for the linear operators $A^{(1)}_\t : = A_\t P_\t$ and $A^{(2)}_\t := A_\t P_\t^\perp$. 
Consequently,  it follows  from \eqref{waveP1} that 
%$U^{(1)}$ and $U^{(2)}$ respectively satisfy
\begin{gather}
	\label{w.U1p}
	(\l P_\t M + \ep^{-1}  A_\t^{(1)} ) U^{(1)} = P_\t (F,0)^\top - P_\t\l  M U^{(2)}, \\
	\label{w.U2p}
	(\l P_\t^\perp M+ \ep^{-1} A_\t^{(2)} ) U^{(2)} = P_\t^\perp  (F,0)^\top- P_\t^\perp \l M  U^{(1)},
\end{gather}
where $M = {\rm diag}(1,a^{-1})$ and $U=U^{(1)}+U^{(2)}$ is decomposed according to $I=P_\t+P_\t^\bot$. Now, from \eqref{waveP2} and \eqref{w.U1p} (see Remark \ref{remEform}) we find
\[
	(\l P_\t M  + \ep^{-1} A_\t^{(1)} ) (U^{(1)} - V) = - P_\t\l M U^{(2)} = P_\t(0,- \l a^{-1} U^{(2)}_2 )^\top.
\]
By noting that $U^{(1)} - V \in E_\t$ and  $A^{(1)}_\t$ is skew-self-adjoint,  one arrives at
\begin{flalign*}
&\mathrm{Re} \big\langle \lambda   M (U^{(1)} - V) , U^{(1)} - V) \big\rangle  \\&= 	\mathrm{Re} \Big( 
\big\langle \lambda  P_\t M (U^{(1)} - V) , U^{(1)} - V \big\rangle + \ep^{-1}\big\langle A^{(1)}_\t (U^{(1)} - V) , U^{(1)} - V \big\rangle   \Big)\\
& = \mathrm{Re}\big\langle  P_\t(0,- \l a^{-1} U^{(2)}_2 )^\top, U^{(1)} - V \big\rangle  \\
&=\mathrm{Re} \big\langle  (0,- \l a^{-1} U^{(2)}_2 )^\top, U^{(1)} - V \big\rangle,
\end{flalign*}
which along with the positivity of the real part of $M =  \mathrm{diag}(1, a^{-1})$ and the boundedness of $a^{-1}$  (see \eqref{coeffass}) implies
\[
\| U^{(1)} - V \|  \lesssim  \|  P_\t(0,- \l a^{-1} U^{(2)}_2 )^\top \| \lesssim   |\l |\|   U^{(2)}_2 \|.
\]
To prove \eqref{W.errest1.1} it remains to show 
\begin{equation}\label{W.U2est}
\| U^{(2)} \| \lesssim \ep|\l|  \| F \|. 
\end{equation}
From \eqref{w.U2p} one has $U^{(2)} = \ep (A^{(2)}_\t)^{-1} ( P_\t^\perp  (F,0)^\top- P_\t^\perp \l M  U)$
%\[
%\ep^{-1} \| A^{(2)}_\t U^{(2)} \|  \lesssim \| P_\t^\perp  (F,0)^\top- P_\t^\perp \l M  U \| \lesssim \| F \| + \| \l U \| \lesssim |\l | \| F\|,
%\]
which along with \eqref{A2ubi} and  the bound 
%where the last inequality used the bound 
$	\|  U  \|  \lesssim \|  F \|$ (which readily follows from considering the real part of \eqref{waveP1}, also cf.~\eqref{coeffass}) implies   \eqref{W.U2est}.
The proof is complete. 
 \end{proof}
The next step in determining the leading-order asymptotics for $U^\ep_{\l,\t}$ is to simplify the $\t$-dependence of $V^\ep_{\l,\t}$ further.
 For this, we shall utilise the following alternative representation.
\begin{proposition}\label{Vrepresentation}
The solution $V^\ep_{\l,\t}$ to \eqref{waveP2} has the form
\begin{equation}\label{Vrep}
V^\ep_{\l,\t} =  c \left( \begin{matrix}
\l  \\ a(I+ \nabla N_\t + \i \t \otimes N_\t)	\frac{\i \t}{\ep } 
\end{matrix} \right), \quad c = (\l^2 + a_\t \tfrac{\t}{\ep} \cdot \tfrac{\t}{\ep} )^{-1}  \langle F,1 \rangle_{L^2(\square)},
\end{equation}
%$c \in \CC $ is the solution to 
%\begin{equation}
%\l c + A^{hom}_\t \frac{\t}{\ep} \cdot \frac{\t}{\ep}  \frac{c}{\l} = \langle F,1 \rangle_{L^2(\square)}
%%\langle a( I + N_\t) \frac{\i \t}{\ep \l}  , \i \tfrac{\t}{\ep} \rangle_{L^2(\square)} = \langle F,1 \rangle_{L^2(\square)} \quad  \l \xi = \i \tfrac{\t}{\ep} c
%\end{equation}
where $a_\t$ is the constant coefficient $d \times d$-matrix given by
\begin{equation}\label{atheta}
(a_\t)_{pq} : = \langle a(I + \nabla N_\t + \i \t \otimes  N_\t) e_p , e_q \rangle_{L^2(\square)}, \qquad p,q \in \{1,2,\ldots,d\},
\end{equation}
for $e_j$ the $j$-th Euclidean basis vector and $N_\t$ the vector whose $j$-th component is the solution $N^j_\t \in H^1_{\per,0}(\square)$ to  \begin{equation}\label{thetacell}
	\langle a(\nabla + \i \t)N^j_\t ,(\nabla + \i \t)\phi \rangle_{L^2(\square)} = -\langle a e_j ,(\nabla + \i \t)\phi \rangle_{L^2(\square)}, \quad (\phi \in H^1_{\per,0}(\square)).
\end{equation}
\end{proposition}
\begin{remark}
	Notice that $N_0$ and  $a_0$ are respectively the classical corrector and homogenised matrix for $a$. 
\end{remark}

\begin{proof}[Proof of Proposition \ref{Vrepresentation}]

From the definition of $E_\t$, see \eqref{E1}, one has \[
V^\ep_{\l,\t} = \left( \begin{matrix}
c_1 \\ c_2 + s	
\end{matrix} \right) 
\] for some $c_1 \in \CC$, $c_2\in  \CC^d$ and $s\in L^2_{\sol, \t}(\square)$. Moreover,  \eqref{waveP2} (cf. Remark \ref{remEform}) states $c_1,c_2$ and $s$ solve the system
\begin{gather}
	\l c_1 - \ep^{-1} \i \t \cdot c_2   = \langle F,1  \rangle_{L^2(\square)},  \label{Vp1}\\
	\l \langle  a^{-1}(c_2+s) ,\tilde{c} \rangle_{L^2(\square;\CC^d)} - \i \tfrac{\t}{\ep}  \langle  c_1,\tilde{c} \rangle_{L^2(\square;\CC^d)}= 0, \quad  (\tilde{c} \in \CC^d) \label{Vp1.1} \\
	\langle a^{-1}(c_2+s) ,\tilde{s} \rangle_{L^2(\square)} = 0, \quad (\tilde{s} \in L^2_{\sol,\t}), \label{Vp2}
\end{gather}
where the third equation used the fact that $\langle \i \t c_1, \tilde{s} \rangle_{L^2(\square)} = 0$ (since $\div\, \tilde{s} +\i \t \cdot \tilde{s} = 0$ in $(H^1_{\per}(\square))^*$). From \eqref{Vp2}, \eqref{maindecomp}$_2$ and \eqref{L2potrep}, it follows that 
\begin{equation}\label{9.7.24e1}
a^{-1}(c_2 + s) = \xi +  (\nabla + \i \t) p_2
\end{equation} for some $\xi \in \CC^d$, $p_2 \in H^1_{\per,0}(\square)$.  Now,  from \eqref{9.7.24e1},  one has
\[
	\langle \xi,\tilde{c} \rangle_{L^2(\square;\CC^d)} =	\langle a^{-1}(c_2 +s),\tilde{c} \rangle_{L^2(\square;\CC^d)}, \quad ( \tilde{c} \in \CC^d),  
\]
which, along with  \eqref{Vp1.1},  gives $\xi = \i \tfrac{\t}{\ep \l} c_1$. Furthermore, from \eqref{9.7.24e1}, one has
%\begin{equation*}
%%	\label{9.7.24e2}
% \langle a\xi  + a(\nabla + \i \t) p_2), (\nabla+\i \t ) \phi \rangle_{L^2(\square;\CC^d)} =\langle c_2, \i \t \phi \rangle_{L^2(\square;\CC^d)}  , \quad (\phi \in H^1_{\mathrm{per}}(\square));
%\end{equation*}
%that is 
\begin{equation*}
%	\label{9.7.24e2}
	\langle a\xi  + a(\nabla + \i \t) p_2), (\nabla+\i \t ) \phi_0\rangle_{L^2(\square;\CC^d)} =0 , \quad (\phi_0 \in H^1_{\mathrm{per},0}(\square)),
\end{equation*}
that is $p_2 = N_\t \cdot \xi$ (cf. \eqref{thetacell}). Thus, the above considerations show that 
\[
V^\ep_{\l,\t} = \left( \begin{matrix}
	c_1 \\ c_2 + s	
\end{matrix} \right)  = \left( \begin{matrix}
c_1 \\ \xi + (\nabla + \i \t) p_2	
\end{matrix} \right)  =\frac{c_1}{\l}  \left( \begin{matrix}
\l \\ \i \frac{\t}{\ep }  + (\nabla + \i \t)	N_\t \cdot \i \frac{\t}{\ep } 
\end{matrix} \right) 
\]
It remains to show that $c = c_1 / \l$ satisfies
\begin{equation}\label{9.7.24e4}
 (\l^2 + a_\t \tfrac{\t}{\ep} \cdot \tfrac{\t}{\ep} ) c = \langle F,1 \rangle_{L^2(\square)}.
\end{equation}
To that end, note that \eqref{9.7.24e1} implies 
\begin{equation*}
%	\label{9.7.24e3}
	\langle a\xi  + a(\nabla + \i \t) p_2), \i \t  \tilde{c} \rangle_{L^2(\square;\CC^d)} =\langle c_2, \i \t \tilde{c} \rangle_{L^2(\square;\CC^d)} , \quad (\tilde{c}\in \CC^d).
\end{equation*} 
This equation and  \eqref{atheta} imply that
\[
c_2 \cdot  \i \t = a_\t \xi  \cdot \i \t,
\]
which, in turn,  along with \eqref{Vp1}, and the identity $\xi = \i \frac{\t}{\ep \l} c_1$, gives \eqref{9.7.24e4}.
\end{proof}

The dependence of $V^\ep_{\l,\t}$ on $\t$ can be further simplified, for small $\ep$,  by utilising the following properties of $a_\t$ and $N_\t$  (established in \cite[Section 4]{CoWa19}): 
\begin{proposition}\label{Ntasym}
Assume $a$ satisfies \eqref{coeffass}. Then, there exists $K>0$ such that  
\begin{equation*}
	%\label{atellipt}
K |\xi|^2 \le a_\t \xi \cdot \overline{\xi} \le K^{-1} |\xi|^2, \quad (\xi \in \RR^d, \,  \t \in \square^*).
\end{equation*}
Furthermore, there exists  $C>0$ such that
\[
%\| N_\t \cdot \xi \|_{H^1(\square)} \le C |\xi|, \qquad 
\| N_\t \xi \| \le C |\xi|, \qquad \| N_\t \cdot \xi - N_0 \cdot  \xi \|_{H^1(\square)} \le C |\t| |\xi |, \quad (\xi \in \CC^d,\, \t \in \square^*),
\]
and 
\[
|a_\t \xi \cdot \overline{\eta} - a_0 \xi \cdot \overline{\eta}| \le C |\t| |\xi | | \eta|, \quad (\xi, \eta \in \CC^d, \,  \t \in \square^*).
\]
\end{proposition}
Now, we are ready for the final step.

\begin{proposition}\label{p.VW}
	For fixed $\ep, \t , \l$ and $F$,  consider the solution $V^\ep_{\l,\t}$ to \eqref{waveP2}  and 
		\begin{equation}
			\label{Wrep1}
		W_{\l,\xi} =  c_0 \left( \begin{matrix}
			\l  \\  a(I+ \nabla N_0 )	\i \xi
		\end{matrix} \right), \quad c_0 = (\l^2 + a_0 \xi \cdot \xi )^{-1} \langle F,1 \rangle_{L^2(\square)}, \quad \xi \in \RR^d.
	\end{equation}
%	 $W_{\l,\xi} \in E_0$ to 
%		\begin{equation}\label{wavePhom}
%		\Bigg(\lambda P_0 \left( \begin{matrix} 1 & 0 \\ 0 & a^{-1}
%		\end{matrix} \right)  +  P_0\left( \begin{matrix} 0 & -(\div + \i \xi \cdot) \\ -(\nabla + \i \xi) & 0
%		\end{matrix} \right) \Bigg)W_{\l,\xi} =P_0  \left( \begin{matrix} F\\ 0
%		\end{matrix} \right).
%	\end{equation} 
 Then, there is $C>0$ independent of $\ep, \l, \t$ and $F$ such that
\begin{equation}
%	\label{W.errest1}
	\| V^\ep_{\l,\t} - W_{\l,\t/\ep} \|_{L^2(\square; \CC^{d+1})} \le C \ep | \l |^2 \| F\|_{L^2(\square)}.
\end{equation}
\end{proposition}
\begin{proof}
 By \eqref{Vrep} and \eqref{Wrep1} one has
\begin{equation*}
\big(V^\ep_{\l,\t}\big)_1 - \big(W_{\l,\t/\ep} \big)_1  = \l (c-c_0),
\end{equation*}
and
\begin{flalign*}
\big(V^\ep_{\l,\t}\big)_2 - \big(W_{\l,\t/\ep} \big)_2 & =a \i \tfrac{\t}{\ep} (c - c_0)  + a\nabla N_0 \i \tfrac{\t}{\ep}   ( c-c_0)   + \Big( a\nabla N_\t \i \tfrac{\t}{\ep} c - a\nabla N_0 \tfrac{\i \t}{\ep}c \Big)+ \i \t \otimes N_\t \tfrac{\i \t}{\ep}c .
\end{flalign*}
The above identities  and Proposition \ref{Ntasym} imply that\marginpar{tilhere}
%	
%%Arguing as in Proposition \ref{Vrepresentation} it is straight-forward to show that
%%	\begin{equation}\label{Wrep}
%%W_{\l,\xi} =  c_0 \left( \begin{matrix}
%%	\l  \\  a(I+ \nabla N_0 )	\i \xi
%%\end{matrix} \right), \quad c_0 = (\l^2 + a_0 \xi \cdot \xi )^{-1} \langle F,1 \rangle_{L^2(\square)}, \quad \xi \in \RR.
%%\end{equation}
%%	
%%By Proposition \ref{prop.W1stest} we need only bound the difference $V^\ep_{\l,\t} - W_{\l,\t/\ep}$. Let us recall the objects of Proposition \ref{Vrepresentation}. 
%Let us show that
%
%\begin{equation}
%	\| V^\ep_{\l,\t} - W_{\l,\t/\ep} \|_{L^2(\square; \CC^{d+1})} \le \ep d(a_\t \tfrac{\t}{\ep} \cdot \tfrac{\t}{\ep}) \| F\|_{L^2(\square)}.
%\end{equation}
%
%where {\color{red}$d(t):=\tfrac{t}{|\l^2 + t|}, t\in [0, \infty)$}
%
%We have 
% 
%\begin{multline*}
%	\| \big(V^\ep_{\l,\t}\big)_2 - \big(W_{\l,\t/\ep} \big)_2\|_{L^2(\square; \CC^{d})} = \|a(I+ \nabla N_\t + \i \t \otimes N_\t)	\tfrac{\i \t}{\ep}c - a(I+ \nabla N_0 )\i \tfrac{\t}{\ep} c_0 \|_{L^2(\square; \CC^{d+1})} \\ = 
%\|a I \i \tfrac{\t}{\ep} (c_\t - c_0) + a(\nabla N_\t c_\t - \nabla N_0 c_0)\tfrac{\i \t}{\ep} + \i \t \otimes N_\t \tfrac{\i \t}{\ep}c \|_{L^2(\square; \CC^{d+1})} 
%\end{multline*} 
%
%Therefore
\[
\| \big(V^\ep_{\l,\t}\big)_1 - \big(W_{\l,\t/\ep} \big)_1 \|_{L^2(\square)}  \le |\l| |c-c_0|, \qquad \text{and} \qquad  \| \big(V^\ep_{\l,\t}\big)_2 - \big(W_{\l,\t/\ep} \big)_2 \|_{L^2(\square; \CC^{d})} 
%\tfrac{|\t|}{\ep} |c_\t-c_0| + \tfrac{|\t|}{\ep} |c_\t-c_0| + \tfrac{|\t|^2}{\ep} |c_\t| 
\lesssim \tfrac{|\t|}{\ep} |c-c_0| +\tfrac{|\t|^2}{\ep} |c|.
\]
By the definitions of $c$,  \eqref{Vrep}, and $c_0$, \eqref{Wrep1}, and Proposition \ref{Ntasym} we compute
\[
\tfrac{|\t|^2}{\ep}|c| \lesssim  \ep \frac{a_\t \tfrac{\t}{\ep }\cdot \tfrac{\t}{\ep}}{|\l^2 +a_\t \tfrac{\t}{\ep }\cdot \tfrac{\t}{\ep}|}\| F\|_{L^2(\square)} = \ep d(a_\t \tfrac{\t}{\ep} \cdot \tfrac{\t}{\ep}) \| F\|_{L^2(\square)},
\]
\begin{flalign*}
	| c - c_0| &= \frac{ | a_\t \tfrac{\t}{\ep} \cdot \tfrac{\t}{\ep} - a_0 \tfrac{\t}{\ep} \cdot \tfrac{\t}{\ep} |}{ |\l^2 + a_\t \tfrac{\t}{\ep} \cdot \tfrac{\t}{\ep} | | \l^2 + a_0 \tfrac{\t}{\ep} \cdot \tfrac{\t}{\ep} | }|\langle F,1\rangle_{L^2(\square)}| 
	\lesssim  \frac{ \tfrac{|\t|^3}{\ep^2} }{ |\l^2 + a_\t \tfrac{\t}{\ep} \cdot \tfrac{\t}{\ep} | | \l^2 + a_0 \tfrac{\t}{\ep} \cdot \tfrac{\t}{\ep} | } \| F \|_{L^2(\square)}
	\\& \lesssim \ep d(a_\t \tfrac{\t}{\ep}\cdot \tfrac{\t}{\ep}) e\Big(\sqrt{a_0 \tfrac{\t}{\ep} \cdot \tfrac{\t}{\ep}}\Big) \| F \|_{L^2(\square)},
\end{flalign*}
and, similarly,
\begin{equation*}
	\tfrac{|\t|}{\ep }| c - c_0| \lesssim \ep d(a_\t \tfrac{\t}{\ep}\cdot \tfrac{\t}{\ep}) d(a_0 \tfrac{\t}{\ep} \cdot \tfrac{\t}{\ep}) \| F \|_{L^2(\square)},
\end{equation*}
where $d(t)\coloneqq\tfrac{t}{|\l^2 + t|}$,  and $e(t)\coloneqq \tfrac{t}{|\l^2 +t^2|}$, for  $t\in [0, \infty)$. Direct calculation shows that
if $\Re(\l^2)>0$ then for $t\in [0,\infty) $, we obtain \[d(t) = \tfrac{t}{\sqrt{(\Re(\l^2) + t)^2 + (\Im \l^2)^2}} \le 1\] and, if $\Re(\l^2)<0$ then $d$ has a positive maximum at $t^* =  \tfrac{-|\l^2|^2}{\Re(\l^2)}$ and, so for all $t\in[0,\infty)$, $d(t) \le d(t^*) =  \tfrac{|\l|^2}{|\Im\l^2 |} \lesssim |\l|$.
Similarly, one can show that
\[
e(t) \le e(|\l|) =  \tfrac{1}{2\nu} \quad (t \in [0,\infty)).
\]
The above results put together yield
\[
\| \big(V^\ep_{\l,\t}\big)_1 - \big(W_{\l,\t/\ep} \big)_1 \|_{L^2(\square)}  \le \ep |\l|^2   \| F \|_{L^2(\square)},  \text{ and } \| \big(V^\ep_{\l,\t}\big)_2 - \big(W_{\l,\t/\ep} \big)_2 \|_{L^2(\square; \CC^{d})} 
%\tfrac{|\t|}{\ep} |c_\t-c_0| + \tfrac{|\t|}{\ep} |c_\t-c_0| + \tfrac{|\t|^2}{\ep} |c_\t| 
\lesssim\ep |\l|^2  \| F \|_{L^2(\square)}.\qedhere
\]\end{proof}
Now, the proof of  Theorem \ref{WaveMainthm} follows from the triangle inequality and Propositions \ref{prop.W1stest} and \ref{p.VW}.

\subsection{Proof of Theorem \ref{whomthm}}\label{ss1.2proof}
\textit{Proof of \eqref{7.6.23m1}.}  For this, we simply use  the following elementary estimate for $\mathcal{P}_\ep $:  there exists $C\geq 0$ such that
\begin{equation}\label{Pe1}
\| h - \mathcal{P}_\ep h \|_{L^2(\RR^d)} \le C \ep\| \nabla h \|_{L^2(\RR^d)}, 
\end{equation}
for all $h \in H^1(\RR^d)$ and all $\ep >0$.

Now, the identity $\Dt \mathcal{P}_\ep u_0 = \mathcal{P}_\ep \Dt u_0$, the above estimate and \eqref{u0reg} gives
\[
	\|    \Dt \mathcal{P}_\ep u_0 - \Dt u_0\|_{L^2_\nu(\RR;L^2(\RR^d))} \lesssim 	\ep\|  \nabla_x \Dt u_0\|_{L^2_\nu(\RR;L^2(\RR^d))}
%	  \lesssim 	\ep\|  \Dt f \|_{L^2_\nu(\RR;L^2(\RR^d))},
\]
which, along with  \eqref{7.7.23e1},  
%and the simple fact\footnote{This follows, for example, from the properties of the unitary Laplace transform $\mathcal{L}_\nu$.}
%\[
%\|  \Dt f \|_{L^2_\nu(\RR;L^2(\RR^d))} \lesssim \|  \Dt^2 f \|_{L^2_\nu(\RR;L^2(\RR^d))}
%\]
implies \eqref{7.6.23m1}.\\

\textit{Proof of  \eqref{7.6.23m2}.}  First note that
	\[
	\ep \nabla(  N_0(\tfrac{\cdot}{\ep}) \cdot \nabla \P_\ep u_0 )=  \nabla N_0(\tfrac{\cdot}{\ep}) \nabla  \P_\ep u_0 +  \ep N_0(\tfrac{\cdot}{\ep}) \nabla^2 \P_\ep u_0.
	\] Thus,
	\begin{align*}
	a(\tfrac{x}{\ep})(I + \nabla N_0(\tfrac{x}{\ep})) \nabla \mathcal{P}_\ep u_0 
%	&=
%	a(\tfrac{x}{\ep})(\nabla \mathcal{P}_\ep u_0 + \nabla N_0(\tfrac{x}{\ep})\nabla \mathcal{P}_\ep u_0) \\
%	& = a(\tfrac{x}{\ep})(\nabla \mathcal{P}_\ep u_0  -\ep N_0(\tfrac{\cdot}{\ep}) \nabla^2 \P_\ep u_0+\ep \nabla(  N_0(\tfrac{\cdot}{\ep}) \cdot \nabla \P_\ep u_0 ))\\
% & =	a(\tfrac{x}{\ep})(\nabla \mathcal{P}_\ep u_0  +\ep \nabla(  N_0(\tfrac{\cdot}{\ep}) \cdot \nabla \P_\ep u_0 )) -a(\tfrac{x}{\ep})\ep N_0(\tfrac{\cdot}{\ep}) \nabla^2 \P_\ep u_0
% \\
 & =	a(\tfrac{x}{\ep})\nabla (\mathcal{P}_\ep u_0  +\ep   N_0(\tfrac{\cdot}{\ep}) \cdot \nabla \P_\ep u_0 ) -a(\tfrac{x}{\ep})\ep N_0(\tfrac{\cdot}{\ep}) \nabla^2 \P_\ep u_0,
	\end{align*} 
which, along with the boundedness of $a$ and $N_0$ (due to the De Giorgi-Nash elliptic regularity theorem) gives  	
	\begin{flalign*}
	 \| a(\tfrac{x}{\ep})(I + \nabla N_0(\tfrac{x}{\ep})) \nabla \mathcal{P}_\ep u_0   - a(\tfrac{\cdot}{\ep} )\nabla ( u_0 + \ep N_0(\tfrac{\cdot}{\ep}) \cdot  \nabla \P_\ep u_0)  \|_{L^2_\nu(\RR;L^2(\RR^d;\RR^d))} \lesssim \ep \| \nabla^2 \mathcal{P}_\ep u_0\|_{L^2_\nu(\RR;L^2(\RR^d;\RR^d))}.
	\end{flalign*}
	Furthermore,  by the properties of the Fourier transform 
	\[
	\|\nabla^2\mathcal{P}_\ep u_0\|_{L^2_\nu(\RR; L^2(\RR^d))}=\|\mathcal{P}_\ep\nabla^2 u_0\| \leq \|\nabla^2 u_0\|_{L^2_\nu(\RR; L^2(\RR^d))}.
%	\lesssim \| \div a_0 \nabla u_0\|
	\]
%	  which along with \eqref{u0reg} gives  
%	\begin{equation}\label{7.7.23e5}
%\|\nabla^2\mathcal{P}_\ep u_0\|_{L^2_\nu(\RR; L^2(\RR^d))}\lesssim \| \Dt^2 f\|_{L^2_\nu(\RR; L^2(\RR^d))}.
%	\end{equation}
	 These above observations, along  with \eqref{7.7.23e1} and \eqref{u0reg}, imply 
	\[
	\| a(\tfrac{\cdot}{\ep}) \nabla u_\ep  - a(\tfrac{\cdot}{\ep} )\nabla ( \P_\ep u_0 + \ep N_0(\tfrac{\cdot}{\ep}) \cdot  \nabla \P_\ep u_0)  \|_{L^2_\nu(\RR;L^2(\RR^d;\RR^d))} \lesssim \ep \| \Dt^2 f \|_{L^2_\nu(\RR;L^2(\RR^d))}.
	\]
Arguing in the same manner as above, we can readily  replace $ \nabla  \P_\ep u_0 $ with $ \nabla  u_0 $ in the above inequality (cf. \eqref{Pe1} and \eqref{u0reg})
	 to get
	 	\begin{equation}\label{7.7.23e7}
	 \| a(\tfrac{\cdot}{\ep}) \nabla u_\ep  - a(\tfrac{\cdot}{\ep} )\nabla ( u_0 + \ep N_0(\tfrac{\cdot}{\ep}) \cdot  \nabla \P_\ep u_0)  \|_{L^2_\nu(\RR;L^2(\RR^d;\RR^d))} \lesssim \ep \| \Dt^2 f \|_{L^2_\nu(\RR;L^2(\RR^d))}.
	 \end{equation}
	In the same manner, changing $\ep \nabla^2 \mathcal{P_\ep}u_0$ by $\ep \nabla^2 u_0$ creates a controlled error of order $\epsilon$.  However,  to replace $\nabla N_0(\tfrac{\cdot}{\ep})   \nabla \P_\ep u_0 $ with $\nabla N_0(\tfrac{\cdot}{\ep})   \nabla  u_0$ is more challenging and  relies on the fact that the corrector $N_0$  is a multiplier in $H^1$ (first shown in \cite[Proposition 8.2]{Su13}):
	\begin{lemma}[Lemma 1.2 of \cite{ZhPa}]There is $C\geq0$ such that, for each $j\in \{1,\ldots,d\}$,
	\begin{equation}\label{multiplerest}
		\int_{\RR^d} |\nabla N^j_0(\tfrac{x}{\ep}) \phi(x)|^2 \, \mathrm{d}x\le C \int_{\RR^d} \big( |\phi(x)|^2 + \ep^2|\nabla \phi(x)|^2 \big) \, \mathrm{d}x, \quad ( \phi \in H^1(\RR^d)).
	\end{equation}
	\end{lemma}
	Thus, 
	\[
	\| \nabla N_0(\tfrac{\cdot}{\ep})   \nabla \P_\ep u_0 - \nabla N_0(\tfrac{\cdot}{\ep})   \nabla  u_0\|_{L^2_\nu(\RR;L^2(\RR^d;\RR^d))} \lesssim \ep \|u_0\|_{L^2_\nu(\RR; H^2(\RR^d))} \lesssim\ep \| \Dt^2 f\|_{L^2_\nu(\RR; L^2(\RR^d))}.
	\]
Hence, \eqref{7.7.23e7} yields
\begin{equation}\label{7.7.23e7b}
		\| a(\tfrac{\cdot}{\ep}) \nabla u_\ep  - a(\tfrac{\cdot}{\ep} )\nabla ( u_0 + \ep N_0(\tfrac{\cdot}{\ep}) \cdot  \nabla  u_0)  \|_{L^2_\nu(\RR;L^2(\RR^d;\RR^d))} \lesssim \ep \| \Dt^2 f \|_{L^2_\nu(\RR;L^2(\RR^d))},
\end{equation}
i.e.  \eqref{7.6.23m2} holds.\\

\textit{Proof of \eqref{7.6.23m3}.} For this we utilise the following result.
\begin{proposition}\label{mv}
	Let $\ep>0$, $\gamma\in L^2_{\per}(\square)$, $\phi\in L^2(\RR^d)$ and assume $\gamma(\tfrac{\cdot}{\ep})\mathcal{P}_\ep \phi\in L^2(\RR^d)$. Then 
	\begin{equation*}
		\mathcal{P}_\ep \big( \gamma(\tfrac{\cdot}{\ep})\mathcal{P}_\ep\phi \big) = \langle \gamma \rangle \mathcal{P}_\ep\phi, 
	\end{equation*}
	where $\langle \phi \rangle = \int_{\square} \phi(y) \, \mathrm{d}y.$ In particular, if, in addition, $\mathrm{supp}\, \FT \phi \subseteq \ep^{-1} \square^*$, then
	\[
	\mathcal{P}_\ep \big( \gamma(\tfrac{\cdot}{\ep})\phi \big) = \langle \gamma \rangle  \phi. 
	\]
\end{proposition} 
\begin{proof}	
	By the Fourier series representation for $L^2_{\mathrm{per}}(\square)$ functions,  $\gamma = \langle \gamma \rangle  + \sum_{n \neq 0} \gamma_n e^{\i 2\pi n \cdot y}$ for some constants $\gamma_n$. 
	Now, since $ \mathcal{P}_\ep \phi = : \psi \in L^2(\RR^d)$ and $ \gamma(\tfrac{\cdot}{\ep})\psi \in L^2(\RR^d)$ then $\psi \sum_{n\neq 0} \gamma_n e^{\i 2\pi n \cdot \frac{x}{\ep}} \in L^2(\RR^d)$ and, consequently,
	\[
		\mathcal{F} (\gamma(\tfrac{\cdot}{\ep}) \psi)= \langle \gamma \rangle   \mathcal{F} \psi + \mathcal{F} \left(  \psi\sum_{n\neq 0} \gamma_n e^{\i 2\pi n \cdot \frac{\cdot}{\ep}} \right).
	\]
	Furthermore, by the properties of the Fourier transform 
	\[
	\mathcal{F} \left(  \psi\sum_{n\neq 0} \gamma_n e^{\i 2\pi n \cdot \frac{\cdot}{\ep}} \right) = \sum_{n\neq 0}  \gamma_n \mathcal{F} \big( \psi e^{\i 2\pi n \cdot \frac{\cdot}{\ep}} \big)  = \sum_{n\neq 0}  \gamma_n \mathcal{F}( \psi)(\cdot - 2 \pi n / \ep) .
		\]
	Now,	 since $\mathcal{F} \psi$ has support in $\ep^{-1} \square^*$ then  $(\mathcal{F}\phi)(\cdot - 2\pi n / \ep)  \equiv 0$ in  $\ep^{-1} \square^*$, for all $n \neq 0$, and
	one has
	\[
		\chi(\ep \cdot)\mathcal{F} \left(  \psi\sum_{n\neq 0} \gamma_n e^{\i 2\pi n \cdot \frac{\cdot}{\ep}} \right) \equiv 0,
	\]
	where we recall that $\chi$ is the characteristic function of $\square^*$. Additionally, 
\[
	\chi(\ep \cdot)  \mathcal{F} \psi =   \mathcal{F} \psi.
\]
Hence, 
	\[
	\chi(\ep \cdot)\mathcal{F} (\gamma(\tfrac{\cdot}{\ep}) \psi)) =  \langle \gamma \rangle    \mathcal{F} \psi= \mathcal{F}( \langle \gamma \rangle  \psi),
	\]
	and (recalling that $\mathcal{P}_\ep = \mathcal{F}^{-1} \chi(\ep \cdot) \mathcal{F}$) taking the inverse Fourier transform completes the proof.
\end{proof}
For each $j\in \{1,\ldots,d\}$, the function $\gamma = a(e_j + \nabla N^j_0)$  gives $\langle \gamma \rangle = a_0 e_j$ and consequently,  by Proposition \ref{mv}  with $\phi =   \partial_{x_j} u_0$,  gives
\[
	\mathcal{P}_\ep \big( a(\tfrac{\cdot}{\ep} ) ( I+ \nabla N_0(\tfrac{\cdot}{\ep})  )  \nabla \mathcal{P}_\ep  u_0  \big)  = a_0 \nabla \mathcal{P}_\ep u_0.
\]
Furthermore,   \eqref{Pe1} and \eqref{u0reg} imply
\[
\|a_0 \nabla  u_0 - a_0 \nabla \mathcal{P}_\ep u_0 \|_{L^2_\nu(\RR;L^2((\RR^d);\CC^d))}  \lesssim \|\Dt^2 f \|_{L^2_\nu(\RR;L^2(\RR^d))},
\]
thus
\begin{equation}\label{7.7.23e10}
\|	\mathcal{P}_\ep \big( a(\tfrac{\cdot}{\ep} ) ( I+ \nabla N_0(\tfrac{\cdot}{\ep})  )  \nabla \mathcal{P}_\ep  u_0  \big) - a_0 \nabla   u_0 \|_{L^2_\nu(\RR;L^2((\RR^d);\CC^d))}  \lesssim \|\Dt^2 f \|_{L^2_\nu(\RR;L^2((\RR^d)))}.
\end{equation}
Since $ a(\tfrac{\cdot}{\ep} ) ( I+ \nabla N_0(\tfrac{\cdot}{\ep})  )  \nabla \mathcal{P}_\ep  u_0  $ only belongs to $L^2$ with the  bound
\begin{equation}\label{7.7.23added}
	\| a(\tfrac{\cdot}{\ep} ) ( I+ \nabla N_0(\tfrac{\cdot}{\ep})  )  \nabla \mathcal{P}_\ep  u_0  \|_{L^2_\nu(\RR;L^2((\RR^d);\CC^d))}  \lesssim \ep \|\Dt^2 f \|_{L^2_\nu(\RR;L^2((\RR^d)))}
\end{equation}
(see \eqref{multiplerest} and \eqref{u0reg}), we must estimate  $I- \mathcal{P}_\ep$ in a weaker norm. Indeed, using the Fourier transform and \eqref{Pe1} one readily has
\begin{equation*}
	%\label{7.7.23e12}
\|h  - \mathcal{P}_\ep h \|_{H^{-1}((\RR^d)))}  \lesssim \ep \| h \|_{L^2(\RR^d)} \quad ( h \in L^2(\RR^d)),
\end{equation*}
which, along with \eqref{7.7.23added}, gives
\begin{multline}\label{7.7.23e11}
	\|	\big( a(\tfrac{\cdot}{\ep} ) ( I+ \nabla N_0(\tfrac{\cdot}{\ep})  )  \nabla \mathcal{P}_\ep  u_0     - 	\mathcal{P}_\ep \big( a(\tfrac{\cdot}{\ep} ) ( I+ \nabla N_0(\tfrac{\cdot}{\ep})  )  \nabla \mathcal{P}_\ep  u_0  \big) \|_{L^2_\nu(\RR;H^{-1}(\RR^d;\CC^d))} \\ \lesssim \ep \|\Dt^2 f \|_{L^2_\nu(\RR;L^2((\RR^d)))}.
\end{multline}
Now, \eqref{7.7.23e1}, \eqref{7.7.23e11} and \eqref{7.7.23e10} imply \eqref{7.6.23m3}.

\section{Heat equation}\label{s.h}
In this section, we explore the techniques just developed in the context of the heat equation. Thus, we consider the heterogeneous heat equation with rapidly oscillating coefficients:
\begin{equation}\label{h1}
	\partial_{t} u_\ep - \div_x ( b(\tfrac{x}{\ep}) \nabla_x u_\ep ) = f, \quad t\in \RR ,\, x\in \RR^d,
\end{equation}
where $\ep>0$ is a fixed parameter (the small period), $b$ is a periodic function with periodicity cell $\square \coloneqq [-\tfrac{1}{2}, \tfrac{1}{2})^d$, $f$ is given and $u_\ep$ is the unknown.

Similar to Section \ref{s.w}, we follow \cite[p 91]{STW22}, to observe that  if $u_\ep$ satisfies \eqref{h1} the vector $U^\ep = ( u_\ep, b(\tfrac{x}{\ep})  \nabla  u_\ep)^\top$ equivalently solves the evolutionary heat system   
\begin{equation}
	\label{hs1}
	\Bigg( \left( \begin{matrix} \Dt & 0 \\ 0 & b^{-1}(\tfrac{x}{\ep})
	\end{matrix} \right) + \left( \begin{matrix} 0 & -\div_x \\ -\nabla_x & 0
	\end{matrix} \right) \Bigg) U^\ep = \left( \begin{matrix} f \\ 0
	\end{matrix} \right) \quad t\in \RR,\, x\in \RR^d.
\end{equation}

Also, as in Section \ref{s.w}, we assume for the rapidly oscillating coefficient $b\in L^\infty(\mathbb{\square}; L(\mathbb{C}^d)$ with $\Re b(y) \xi \cdot \overline{\xi} \ge \kappa |\xi|^2$ for some $\kappa>0$ and all $\xi \in \RR^d$ and (almost all) $y \in \square$
or, equivalently, the same estimates for $b(\cdot)^{-1}$, see Remark \ref{rem:coerc}.

If $f \in L^2_\nu(\RR; L^2(\RR^d))$ then by \cite[Theorem 6.2.4]{STW22} for each $\nu>0$, there exists a unique solution\footnote{In fact,  $U^\ep$ is known to have additional regularity, see Lemma \ref{p.reg} below.} $U^\ep \in H^1_\nu(\RR; L^2(\RR^d)) \cap L^2_\nu(\RR; {\rm dom}\,  \left( \begin{matrix} 0 & -\div \\ -\nabla & 0
\end{matrix} \right))$ to the evolutionary heat system \eqref{hs1}.

We shall argue as in Section \ref{s.wass} to prove, below, the following homogenisation theorem. 

\begin{theorem}\label{hEstthm}
	Let $\ep>0, \t \in \square^* , \l=\nu+ik, k\in \RR$,  and $F\in L^2_{\per}(\square; \CC \times \CC^{d})$. Consider the solution 
	$U^\ep_{\l,\t}$
	to 
	\begin{equation}\label{heatP1}
		\Bigg(  \left( \begin{matrix} \l & 0 \\ 0 & b^{-1}
		\end{matrix} \right) + \ep^{-1} \left( \begin{matrix} 0 & -(\div + \i \t \cdot) \\ -(\nabla + \i \t) & 0
		\end{matrix} \right) \Bigg) U^\ep_{\l,\t}= \left( \begin{matrix} F \\ 0
		\end{matrix} \right),
	\end{equation} and 
	\begin{equation}\label{hp.Wrep}
		W_{\l,\xi} =  c_0 \left( \begin{matrix}
			1 \\  b(I+ \nabla N_0 )	\i \xi
		\end{matrix} \right), \quad c_0 = (\l + b_0 \xi \cdot \xi )^{-1} \langle F,1 \rangle_{L^2(\square)}, \quad \xi \in \RR.
	\end{equation}
	Then, there exists $C>0$ uniformly in $\ep, \l, \t$ and $F$ such that
	\begin{equation}
		\label{h.mainest}
		\| \l^{1/2} (U^\ep_{\l,\t})_1 - \l^{1/2}(W_{\l,\t/\ep})_1 \|_{L^2(\square)}  +		\| (U^\ep_{\l,\t})_2 - (W_{\l,\t/\ep})_2 \|_{L^2(\square; \CC^{d})} \le C \ep \| F\|_{L^2(\square)}.
	\end{equation}
\end{theorem}

Upon setting $F : =  \GT  \RS \FL f$ in Theorem \ref{hEstthm}, one has $U^\ep : = \FL^{-1} \Gamma^{-1}_\ep \GT^{-1} U^{\ep}_{\l,\t}$ (see Section \ref{s.w} for details)  and, arguing exactly as in Section \ref{ss1.2proof}, we  establish that
\[
\FL^{-1} \Gamma^{-1}_\ep \GT^{-1} W_{\l,\t/\ep} = \left( \begin{matrix}
	\P_\ep u_0 \\ b(\tfrac{\cdot}{\ep}) (I+ \nabla N_0(\tfrac{\cdot}{\ep})) \nabla \P_\ep u_0
\end{matrix} \right)
\]
for $u_0$ being the solution to the homogenised heat equation
	\begin{equation}\label{hhomeqn}
	\Dt u_0 - {\rm div}_x b_0 \nabla_x u_0 = f,
\end{equation}
where $b_0$ is the homogenised coefficients for $b$. Then, via the properties of $\P_\ep$, $N_0$ and $u_0$ we readily  establish Theorem \ref{hhomthm}.
\begin{theorem}\label{hhomthm}
	Fix $\ep>0$ and $f \in L^2_\nu(\RR; L^2(\RR^d))$,  and suppose $b$ satisfies \eqref{coeffass} replacing $a$.	Then $u_\ep$ the solution to \eqref{h1} and  $u_0$ the solution to \eqref{hhomeqn}	satisfy:
	\begin{align}
		\| \Dt^{1/2} u_\ep - \Dt^{1/2}  u_0 \|_{L^2_\nu(\RR;L^2(\RR^d))} \le C \ep \| f\|_{L^2_\nu(\RR;L^2(\RR^d))}, \label{th 2.2.1}\\
		\| b(\tfrac{\cdot}{\ep}) \nabla u_\ep  - b(\tfrac{\cdot}{\ep} )\nabla ( u_0 + \ep N_0(\tfrac{\cdot}{\ep}) \cdot  \nabla  u_0)  \|_{L^2_\nu(\RR;L^2(\RR^d;\RR^d))} \le C \ep \|  f \|_{L^2_\nu(\RR;L^2(\RR^d))}, \label{th 2.2.2}\\
		\| b(\tfrac{\cdot}{\ep}) \nabla u_\ep  - b_0   \nabla  u_0 \|_{L^2_\nu(\RR;H^{-1}(\RR^d;\RR^d))} \le C \ep \| f \|_{L^2_\nu(\RR;L^2(\RR^d))}\label{th 2.2.3}, 
	\end{align}
	for some $C>0$ independent of $\ep$ and $f$. 
\end{theorem}
Similarly to the case of the wave equation, we employ that  $\Dt$ is boundedly invertible in $L^2_\nu(\RR^d; L^2(\RR^d))$ for $\nu>0$, \eqref{th 2.2.1}, that (the real part of) $b$ is uniformly positive with \eqref{th 2.2.2} as well as \eqref{th 2.2.1} and \eqref{th 2.2.3} to deduce the following set of estimates.
\begin{corollary}Under the assumptions of Theorem \ref{hhomthm},
\begin{align*}
\|   u_\ep - u_0  \|_{L^2_\nu(\RR;L^2(\RR^d))} &\le C\ep \| \Dt^{-1/2} f\|_{L^2_\nu(\RR;L^2(\RR^d))},\\
\|   u_\ep - (u_0 + \ep N_0(\tfrac{\cdot}{\ep}) \cdot \nabla u_0) \|_{L^2_\nu(\RR;H^1(\RR^d))} &\le C\ep \| f\|_{L^2_\nu(\RR;L^2(\RR^d))} \\
\Big(\|   U^\ep_1 - U^0_1  \|_{L^2_\nu(\RR;L^2(\RR^d))}^2 + 	\| U^\ep_2 - U^0_2 \|_{L^2_\nu(\RR;H^{-1}(\RR^d;\RR^d))}^2\Big)^{1/2} & \le C \ep \| f\|_{L^2_\nu(\RR;L^2(\RR^d))},
\end{align*}
with $C\geq 0$ independent of $\ep$ and $f$.
\end{corollary}

Next, we turn to the proofs of the main results of this section.

\subsection{Proof of Theorem \ref{hEstthm}}\label{s.hass}

The proof of Theorem \ref{hEstthm}, in essence, goes as in Section \ref{s.wass} with the only difference being that one has better estimates in $\l$ due to the improved regularity of the solution to the evolutionary heat system compared to that of the evolutionary wave system. Indeed, these better estimates will rely on the  following maximal regularity result for evolutionary parabolic equations, first established and proved in \cite[Section 3.1]{TrWa21}; see also \cite[Example 15.24.4]{STW22} or \cite{PTW17} for a first account on maximal regularity for evolutionary equations in the sense of \cite{STW22}.
\begin{lemma}\label{p.reg}
	Let $H_1, H_2$ be  Hilbert spaces, suppose $C: {\rm dom}\, C \subseteq H_2 \rightarrow H_1$ be a densely defined linear operator, and $\mathfrak{n} \in L({H_2})$ satisfy
	\begin{equation}\label{L.h1}
		% {\rm Re}\, \langle \mathfrak{m} \phi, \phi  \rangle_{\mathcal{H}} \ge \kappa_1 \| \phi \|^2_{\mathcal{H}},  \quad
		{\rm Re}\, \langle \mathfrak{n} \phi, \phi  \rangle_{{H_2}} \ge \kappa \| \phi \|^2_{{H_2}},  \qquad  (\phi \in {H_2}),
	\end{equation}
for some $\kappa>0$. Then, for fixed $\mathfrak{f} = (f_1,f_2) \in {H}_1\times H_2$, $ \l = \nu + i k, k \in \RR$,  the solution $U= (u_1,u_2)$ to 
\[
\Bigg( \left( \begin{matrix} \l & 0 \\ 0&\mathfrak{n} 
\end{matrix} \right) + \left( \begin{matrix} 0 & C \\-C^* & 0
\end{matrix} \right)   \Bigg) U = \mathfrak{f},
\]
satisfies
	\begin{multline}
		\label{Preg}
		\| \l u_1 \|_{H_1} + \|\l^{1/2} u_2 \|_{H_2} +  \|\l^{1/2} C^* u_1 \|_{H_2}  +  \| C u_2 \|_{H_1} \\ \lesssim \| f_1 \|_{H_1} + \| \l^{1/2} f_2 \|_{H_2}, \quad (\l = \nu + i k, k \in \RR).
	\end{multline}
%	Here $\lesssim$ means LHS bounded from above by RHS by a positive constant independent of $\ep, \l, \t$ and $F$.
\end{lemma}
Let us  prove Theorem \ref{hEstthm}. For this recall $E_\theta$ from \eqref{E1}. The analogue of Proposition \ref{prop.W1stest} is as follows:
\begin{proposition}\label{prop.P1stest}
	Consider $U^\ep_{\l,\t}$ the solution to \eqref{heatP1}  and $V^\ep_{\l,\t}$ the solution to 
	\begin{equation}\label{heatP2}
		\Bigg(P_\t \left( \begin{matrix} \l  & 0 \\ 0&b^{-1} 
		\end{matrix} \right)   + \ep^{-1}  \left( \begin{matrix} 0 & -(\div + \i \t \cdot) \\ -(\nabla + \i \t) & 0
		\end{matrix} \right) \Bigg)V^\ep_{\l,\t} =P_\t  \left( \begin{matrix} F\\ 0
	\end{matrix} \right),
	\end{equation} 
	for
	%  $M_\t = \P_\t M \P_\t$, $N_\t = \P_\t N \P_\t$ and $A^{(1)}_{\t} = \P_\t A_\t $ where 
	$P_\t : L^2(\square;\CC^{d+1}) \rightarrow E_\t$ the orthogonal projection. Then, the following estimate holds: 
	\begin{equation}\label{hP.errest1}
		\| \l^{1/2} (U^\ep_{\l,\t})_1 - \l^{1/2} (V^\ep_{\l,\t})_1 \|_{L^2(\square)} + 	\| (U^\ep_{\l,\t})_2 - (V^\ep_{\l,\t})_2 \|_{L^2(\square;\CC^d)}\lesssim \ep  \| F\|_{L^2(\square)}.
	\end{equation}
\end{proposition}
\begin{proof}
Dropping super and subscripts; from the assumptions on $b$, it follows from  Lemma \ref{p.reg} that the solution $U$ to \eqref{heatP1} satisfies 
	\begin{equation}\label{hPerr.p1}
		\|  \l U_1 \|_{L^2(\square)}+ \| \l^{1/2} U_2  \|_{L^2(\square;\CC^d)} \lesssim \|  F \|_{L^2(\square)}. 
	\end{equation}
As in the proof of Proposition \ref{prop.W1stest},   $U^{(1)} := P_\t U \in E_\t$, $U^{(2)} :=  P^\perp_\t U\in E^\perp_\t$ are readily seen to solve 
	\begin{equation}\label{hP.U1p}
		%\left. \begin{aligned}
		( P_\t M + \ep^{-1} P_\t A_\t ) U^{(1)} = P_\t (F,0)^\top - P_\t M U^{(2)},
		\end{equation}
	and 
	\begin{equation}\label{P.U2p}
		%\left. \begin{aligned}
		(P^\perp_\t M+ \ep^{-1} P^\perp_\t A_\t ) U^{(2)} = P^\perp_\t (F,0)^\top- P^\perp_\t M U^{(1)},
	\end{equation}
	where $M=  \left( \begin{matrix} \l & 0 \\  0&b^{-1}
	\end{matrix} \right)$. Now, from \eqref{heatP2} and \eqref{hP.U1p}
	it follows that (cf. Remark \ref{remEform})
		\begin{equation}\label{heatdifference}
	(P_\t M + \ep^{-1} P_\t A_\t ) (U^{(1)} - V) = - P_\t M U^{(2)} = P_\t(0,- b^{-1} U^{(2)}_2 )^\top.
	\end{equation}
Notice that \eqref{heatdifference} is a well-posed evolutionary problem on $E_\t$ and an application of  Lemma \ref{p.reg} (for  $H^1 = \CC$, $H^2=\CC^d \oplus L^2_{\sol,\t}(\square)$, $\mathfrak{n} = P_2 b^{-1}$ and $C = -\ep^{-1}  (\div + \i \t \cdot) P_1$ with $C^* =  -\ep^{-1} (\grad + \i \t )P_2$, see Remark \ref{remEform}) gives 
	\begin{equation}\label{30.06.21e1}
	\| \l U^{(1)}_1 - \l V_1 \|_{L^2(\square)}+ 	\| \l^{1/2} U^{(1)}_2 - \l^{1/2} V_2 \|_{L^2(\square;\CC^d)}  \lesssim  \|  \l^{1/2}  U^{(2)}_2 \|_{L^2(\square)}.
	\end{equation}
Similarly, \eqref{P.U2p} is a well-posed evolutionary problem on $E^\perp_\t$and  Lemma \ref{p.reg}, along with the bound \eqref{hPerr.p1} and the fact $P^\perp_\t M U^{(1)}=P^\perp_\t (0,b^{-1}U^{(1)}_2)^\top $, gives
		\[
	\|\l^{1/2} \ep^{-1} (\nabla + \i \t) U^{(2)}_1 \|_{L^2(\square)} + \| \ep^{-1} (\div + \i \t) U^{(2)}_2 \|_{L^2(\square);\CC^d} \lesssim \| F \|_{L^2(\square)} + \| \l^{1/2} U_2 \|_{L^2(\square;\CC^d)} \lesssim \| F\|_{L^2(\square)}.
	\]
Consequently, by utilising \eqref{A2ubi} and the above inequality, it follows that 
	\begin{equation}\label{hP.U2est}
		\| \l^{1/2} U^{(2)}_1 \|_{L^2(\square)} + \| U^{(2)}_2 \|_{L^2(\square;\CC^d)} \lesssim \ep  \| F \|_{L^2(\square)}. 
	\end{equation}
It is straight-forward to deduce that \eqref{hP.errest1}  follows from  \eqref{30.06.21e1} and \eqref{hP.U2est}. The proof is complete. 
\end{proof}
Next, we provide the analogue of Proposition \ref{Vrepresentation} (with the proof being essentially the same).
\begin{proposition}\label{hVrepresentation}
	The solution $V^\ep_{\l,\t}$ to \eqref{heatP2} has the form
	\[
	V^\ep_{\l,\t} =  c \left( \begin{matrix}
		1 \\ b(I+ \nabla N_\t + \i \t \otimes N_\t)	\frac{\i \t}{\ep } 
	\end{matrix} \right), \quad c = (\l + b_\t \tfrac{\t}{\ep} \cdot \tfrac{\t}{\ep} )^{-1}  \langle F,1 \rangle_{L^2(\square)}.
	\] 
\end{proposition}
The final step in the proof of Theorem \ref{hEstthm} is to provide the analogue of Proposition \ref{p.VW}.
\begin{proposition}\label{hp.VW}
	For fixed $\ep, \t , \l$,  consider the solution $V^\ep_{\l,\t}$ to \eqref{heatP2}  and 
	\begin{equation*}
		%			\label{Wrep}
		W_{\l,\xi} =  c_0 \left( \begin{matrix}
			1 \\  b(I+ \nabla N_0 )	\i \xi
		\end{matrix} \right), \quad c_0 = (\l + b_0 \xi \cdot \xi )^{-1} \langle F,1 \rangle_{L^2(\square)}, \quad \xi \in \RR.
	\end{equation*}
	
	Then, there exists $C>0$ independent of $\ep, \l, \t$ and $f$ such that
	\begin{equation}
			\label{h.errest1}
	\| \l^{1/2} (V^\ep_{\l,\t})_1 - \l^{1/2}(W_{\l,\t/\ep})_1 \|_{L^2(\square)}  +		\| (V^\ep_{\l,\t})_2 - (W_{\l,\t/\ep})_2 \|_{L^2(\square; \CC^{d})} \le C \ep \| F\|_{L^2(\square)}.
	\end{equation}
\end{proposition}
\begin{proof}
One has
\begin{equation*}
	\big(V^\ep_{\l,\t}\big)_1 - \big(W_{\l,\t/\ep} \big)_1  = c-c_0,
\end{equation*}
and
\begin{flalign*}
	\big(V^\ep_{\l,\t}\big)_2 - \big(W_{\l,\t/\ep} \big)_2 & =b \i \tfrac{\t}{\ep} (c - c_0)  + b\nabla N_0 \i \tfrac{\t}{\ep}   ( c-c_0)   + b\big( \nabla N_\t  - \nabla N_0  \big)\i \tfrac{\t}{\ep} c+ \i \t \otimes N_\t \tfrac{\i \t}{\ep}c .
\end{flalign*}
The above identities  and \eqref{coeffass} imply that
		\[
		\| \big(V^\ep_{\l,\t}\big)_1 - \big(W_{\l,\t/\ep} \big)_1 \|_{L^2(\square)}  \le |c-c_0|, \text{ and }  \| \big(V^\ep_{\l,\t}\big)_2 - \big(W_{\l,\t/\ep} \big)_2 \|_{L^2(\square; \CC^{d})} 
		%\tfrac{|\t|}{\ep} |c_\t-c_0| + \tfrac{|\t|}{\ep} |c_\t-c_0| + \tfrac{|\t|^2}{\ep} |c_\t| 
		\lesssim \tfrac{|\t|}{\ep} |c-c_0| +\tfrac{|\t|^2}{\ep} |c|.
		\]
		By the definitions of $c$ and $c_0$, and Proposition \ref{Ntasym} we compute
		
		\[
		\tfrac{|\t|^2}{\ep}|c| \lesssim  \ep \frac{b_\t \tfrac{\t}{\ep }\cdot \tfrac{\t}{\ep}}{|\l +b_\t \tfrac{\t}{\ep }\cdot \tfrac{\t}{\ep}|}\| F\|_{L^2(\square)} = \ep d(b_\t \tfrac{\t}{\ep} \cdot \tfrac{\t}{\ep}) \| F\|_{L^2(\square)},
		\]
		
		\begin{flalign*}
			| c - c_0| &= \frac{ | b_\t \tfrac{\t}{\ep} \cdot \tfrac{\t}{\ep} - b_0 \tfrac{\t}{\ep} \cdot \tfrac{\t}{\ep} |}{ |\l + b_\t \tfrac{\t}{\ep} \cdot \tfrac{\t}{\ep} | | \l + b_0 \tfrac{\t}{\ep} \cdot \tfrac{\t}{\ep} | }|\langle F,1\rangle_{L^2(\square)}| 
			\lesssim  \frac{ \tfrac{|\t|^3}{\ep^2} }{ |\l + b_\t \tfrac{\t}{\ep} \cdot \tfrac{\t}{\ep} | | \l + b_0 \tfrac{\t}{\ep} \cdot \tfrac{\t}{\ep} | } \| F \|_{L^2(\square)}
			\\& \lesssim \ep d(b_\t \tfrac{\t}{\ep}\cdot \tfrac{\t}{\ep}) e\Big(\sqrt{b_0 \tfrac{\t}{\ep} \cdot \tfrac{\t}{\ep}}\Big) \| F \|_{L^2(\square)},
		\end{flalign*}
		and, similarly,
		\begin{equation*}
			\tfrac{|\t|}{\ep }| c - c_0| \lesssim \ep d(b_\t \tfrac{\t}{\ep}\cdot \tfrac{\t}{\ep}) d(b_0 \tfrac{\t}{\ep} \cdot \tfrac{\t}{\ep}) \| F \|_{L^2(\square)},
		\end{equation*}
		where $d(t):=\tfrac{t}{|\l + t|}$,  and $e(t) : = \tfrac{t}{|\l +t^2|}$, for  $t\in [0, \infty)$. With $\lambda=\nu+ \i k$	
		it is straightforward to see that $d(t) = t\big((\nu + t)^2 + k^2 \big)^{-1/2}\le 1$ for all $t \in [0,\infty)$. 
		One can readily show that $e(t) = t\big((\nu + t^2)^2 + k^2 \big)^{-1/2} \le e(|\l|^{-1/2}) \lesssim  |\l|^{-1/2}$ for all $t \in [0,\infty)$.	
		Putting the above results together gives
		\[
		\| \big(V^\ep_{\l,\t}\big)_1 - \big(W_{\l,\t/\ep} \big)_1 \|_{L^2(\square)}  \le \ep |\l| ^{-1/2}  \| F \|_{L^2(\square)},\text{ and }  \| \big(V^\ep_{\l,\t}\big)_2 - \big(W_{\l,\t/\ep} \big)_2 \|_{L^2(\square; \CC^{d})} 
		%\tfrac{|\t|}{\ep} |c_\t-c_0| + \tfrac{|\t|}{\ep} |c_\t-c_0| + \tfrac{|\t|^2}{\ep} |c_\t| 
		\lesssim\ep \| F \|_{L^2(\square)},
		\]			which implies \eqref{h.errest1} and completes the proof.
\end{proof}

The remaining claims of Section \ref{s.h} can be shown in a completely analogous manner as the respective ones for the wave system.

\section{Thermo-elastic system of equations} \label{s.te}

The final application of our methods to show estimates in homogenisation employing the framework of evolutionary equations concerns thermo-elasticity. In fact, we combine our results for the wave system and the heat system provided earlier. 

Consider the system of thermo-elastic equations with rapidly oscillating coefficients
\begin{equation}\label{orte}
	\left\{ \begin{aligned}
		\partial_{t}^2 u_\ep - \div_x \big( a(\tfrac{x}{\ep}) \nabla_x u_\ep \big)+ \gamma(\tfrac{x}{\ep}) v_\ep = f, \qquad t\in \RR, \, x \in \RR^d ,\\
		\d_t v_\ep - \div_x\big( b(\tfrac{x}{\ep}) \nabla_x v_\ep \big) - \overline{\gamma}(\tfrac{x}{\ep}) \d_t u_\ep = g,\qquad t\in \RR, \, x \in \RR^d.
	\end{aligned} \right.
\end{equation}

here, $\ep >0$, $u_\ep,v_\ep$ are the unknown solutions to be determined, $f$, $g$ are given, the coefficients $a$, $b$, are matrix-valued periodic functions and $\gamma$ is a scalar-valued  periodic function; all  with fundamental periodicity cell $\square: =[-\tfrac{1}{2}, \tfrac{1}{2})^d$. We additionally assume $x\mapsto \gamma(\tfrac{x}{\ep})$ (and therefore $x\mapsto\overline{\gamma}(\tfrac{x}{\ep})$) is a multiplier in $H^1$, i.e., there is $C\geq0$ such that
	\begin{equation}\label{multiplerest2}
	\int_{\RR^d} |(\nabla \gamma)(\tfrac{x}{\ep}) \phi(x)|^2 \, \mathrm{d}x\le C \int_{\RR^d} \big( |\phi(x)|^2 + |\nabla \phi(x)|^2 \big) \, \mathrm{d}x, \quad (\phi \in H^1(\RR^d)).
\end{equation}

We shall compare \eqref{orte} to the homogenised system 
\begin{equation}\label{orhte}
	\left\{ \begin{aligned}
		\partial_{t}^2 u_0 - \div_x a_0 \nabla_x u_0 + \langle \gamma \rangle v_0 = f, \qquad t\in \RR, \, x \in \RR^d ,\\
		\d_t v_0 - \div_x b_0 \nabla_x v_0  - \overline{\langle \gamma \rangle}\Dt u_0 = g,\qquad t\in \RR, \, x \in \RR^d.
	\end{aligned} \right.
\end{equation}
We begin by rewriting \eqref{orte} and \eqref{orhte} as  a system of evolutionary equations as follows: for $\ep>0$ the unknown $U^\ep=(\partial_t u_\ep, a(\tfrac{x}{\ep})\nabla_x u_\ep, v_\ep,b(\tfrac{x}{\ep})\nabla_x v_\ep)^\top$ solves 
\begin{equation}\label{EvoP}
	(\partial_t M_\ep+N_\ep+A)U^\ep = F,  
\end{equation} 
where $F = (f,0,g,0)^\top$ and
\begin{multline*}
M_\ep= \begin{pmatrix}
	1 & 0 & 0 & 0 \\ 0 & a^{-1}(\tfrac{x}{\ep}) & 0 & 0 \\ 0 & 0 & 1 & 0 \\ 0 & 0 & 0 &0 
\end{pmatrix}, \  N_\ep = \begin{pmatrix}
	0 & 0 & \gamma(\tfrac{x}{\ep})  & 0 \\ 0 & 0 & 0 & 0 \\ -\overline{\gamma}(\tfrac{x}{\ep}) & 0 & 0 & 0 \\ 0 & 0 & 0 &b^{-1}(\tfrac{x}{\ep})
\end{pmatrix} \text{ and} \\  A = \begin{pmatrix}
	0 & -\div_x  & 0 &0 \\  -\nabla_x  & 0 &  0 &  0 \\ 0 & 0 & 0 & -\div_x \\ 0 & 0 & -\nabla_x &0 
\end{pmatrix}.
\end{multline*}
For $\ep =0 $,  $U^0 = (\Dt u_0, a_0 \nabla_x u_0, v_0, b_0\nabla_x v_0)^\top$  solves \eqref{EvoP}  with
\[
M_0 = \begin{pmatrix}
	1 & 0 & 0 & 0 \\ 0 & a^{-1}_0 & 0 & 0 \\ 0 & 0 & 1 & 0 \\ 0 & 0 & 0 &0 
\end{pmatrix}\  \quad{and} \ \  N_0 = \begin{pmatrix}
	0 & 0 & \langle \gamma \rangle & 0 \\ 0 & 0 & 0 & 0 \\ -\overline{\langle \gamma \rangle} & 0 & 0 & 0 \\ 0 & 0 & 0 &b^{-1}_0
\end{pmatrix}.
\]

The following well-posedness result for \eqref{EvoP} is well-known (see {\cite{P09} or \cite{STW22,W13}). 

\begin{theorem}\label{evte} Fix $\ep\ge 0$.  Suppose, $a,b,\gamma \in L^\infty(\RR^d;L(\CC^d))$ be such that
\[
  a(x)=a(x)^* \text{ and }a(x) \geq \kappa_1\text{ and }\Re b(x)\geq \kappa_2
\] in the sense of positive definiteness for a.e.~$x\in \RR^d$ for some $\kappa_1,\kappa_2>0$. 

	Then, for all  $\nu>0$ and for all $f\in H^1_{\nu}(\RR;L^2(\RR^d))$ and $g\in L^2_{\nu}(\RR;L^2(\RR^d))$ there exists a unique solution $U^\ep $ to \eqref{EvoP} with 
	$U^\ep_1, U^\ep_3  \in H_\nu^1(\RR;L^2(\RR^d))\cap L^2_{\nu}(\RR; H^1(\RR^d));$ $U^\ep_2  \in H_\nu^1(\RR;L^2(\RR^d ;\CC^d))\cap L^2_{\nu}(\RR; H_{\rm div}(\RR^d))$ and 
	$U^\ep_4  \in  L^2_{\nu}(\RR;  H_{\rm div}(\RR^d))$. Furthermore, the following inequalities hold: 
	\begin{multline}\label{te.e0}
\| U_1^\ep \|_{H_\nu^1(\RR;L^2(\RR^d))} + \|U_2^\ep \|_{H_\nu^1(\RR;L^2(\RR^d ;\CC^d))\cap L^2_{\nu}(\RR; H_{\rm div}(\RR^d))} + \|U_3^\ep \|_{H_\nu^1(\RR;L^2(\RR^d))} + \|U_4^\ep\|_{ L^2_{\nu}(\RR;  H_{\rm div}(\RR^d))} \\ \lesssim \| f\|_{H_\nu^1(\RR;L^2(\RR^d))} + \|g\|_{ L^2_{\nu}(\RR; L^2(\RR^d))}.
	\end{multline}
\end{theorem}
In this section, we shall prove the following homogenisation theorem:
\begin{theorem}\label{tehomthm}
	Fix $\ep> 0$ and $f \in H^2_\nu(\RR; L^2(\RR^d))$ and  $g \in H^1_\nu(\RR; L^2(\RR^d))$. Suppose \eqref{evte}  holds. 
	Then, the following inequalities hold:
	\begin{align}
	&\notag	\| \Dt u_\ep - \Dt u_0 \|_{L^2_\nu(\RR;L^2(\RR^d))} + 	\| \Dt^{1/2} v_\ep - \Dt^{1/2} v_0 \|_{L^2_\nu(\RR;L^2(\RR^d))} 	\\&+	\| a(\tfrac{\cdot}{\ep}) \nabla u_\ep  - a(\tfrac{\cdot}{\ep} )\nabla ( u_0 + \ep N_0(\tfrac{\cdot}{\ep}) \cdot  \nabla  u_0)  \|_{L^2_\nu(\RR;L^2(\RR^d;\RR^d))} \nonumber \\
	&	+ 	\| b(\tfrac{\cdot}{\ep}) \nabla v_\ep  -b(\tfrac{\cdot}{\ep} )\nabla ( v_0 + \ep N_0(\tfrac{\cdot}{\ep}) \cdot  \nabla  v_0)  \|_{L^2_\nu(\RR;L^2(\RR^d;\RR^d))} 
	+	\| a(\tfrac{\cdot}{\ep}) \nabla u_\ep  - a_0   \nabla  u_0 \|_{L^2_\nu(\RR;H^{-1}(\RR^d;\RR^d))}  \nonumber\\
	&+ 	\| b(\tfrac{\cdot}{\ep}) \nabla v_\ep  - b_0   \nabla  v_0 \|_{L^2_\nu(\RR;H^{-1}(\RR^d;\RR^d))}  \le C \ep \big( \| \Dt^2 f \|_{L^2_\nu(\RR;L^2(\RR^d))} +  \| \Dt  g\|_{L^2_\nu(\RR;L^2(\RR^d))} \big), \label{temain1}
	\end{align}
	for some $C>0$ independent of $\ep$, $f$ and $g$.
\end{theorem}

\begin{proof}[Proof of Theorem \ref{tehomthm}]
Note that $u_\ep$ solves the wave equation
\begin{equation}\label{te.e1}
	\partial_{tt} u_\ep - \div_x \big( a(\tfrac{x}{\ep}) \nabla_x u_\ep \big) = f - \gamma(\tfrac{x}{\ep}) v_\ep , \qquad t\in \RR, \, x \in \RR^d.\\
\end{equation}
By Theorem \ref{evte}, one has 
\[
\| \Dt^2 v_\ep \|_{L^2_\nu(\RR;L^2(\RR^d))} \lesssim \| \Dt^2 f \|_{L^2_\nu(\RR;L^2(\RR^d))} +\| \Dt g \|_{L^2_\nu(\RR;L^2(\RR^d))} 
\]
which, along with the boundedness of $\gamma$, implies
\[
 \| \Dt^2 (f- \gamma(\tfrac{x}{\ep}) v_\ep )\|_{L^2_\nu(\RR;L^2(\RR^d))} \lesssim  \| \Dt^2 f \|_{L^2_\nu(\RR;L^2(\RR^d))} +  \| \Dt  g\|_{L^2_\nu(\RR;L^2(\RR^d))}.
\]
From this observation, along with Theorem \ref{whomthm}, one has 
	\begin{align}
	\| \Dt u_\ep - \Dt u^\ep_0 \|_{L^2_\nu(\RR;L^2(\RR^d))} &\le C \ep \big( \| \Dt^2 f \|_{L^2_\nu(\RR;L^2(\RR^d))} +  \| \Dt  g\|_{L^2_\nu(\RR;L^2(\RR^d))} \big) \label{te.e2}\\
	\| a(\tfrac{\cdot}{\ep}) \nabla u_\ep  - a(\tfrac{\cdot}{\ep} )\nabla ( u^\ep_0 + \ep N_0(\tfrac{\cdot}{\ep}) \cdot  \nabla  u^\ep_0)  \|_{L^2_\nu(\RR;L^2(\RR^d;\RR^d))} &\le C \ep \big( \| \Dt^2 f \|_{L^2_\nu(\RR;L^2(\RR^d))} +  \| \Dt  g\|_{L^2_\nu(\RR;L^2(\RR^d))} \big)\label{te.e3} \\
	\| a(\tfrac{\cdot}{\ep}) \nabla u_\ep  - a_0   \nabla  u^\ep_0 \|_{L^2_\nu(\RR;H^{-1}(\RR^d;\RR^d))}& \le C \ep\big( \| \Dt^2 f \|_{L^2_\nu(\RR;L^2(\RR^d))} +  \| \Dt  g\|_{L^2_\nu(\RR;L^2(\RR^d))} \big)\label{te.e4} 
\end{align}
for $u^\ep_0$ the solution to 
\begin{equation}\label{te.e5}
	\partial_{tt} u^\ep_0 +\div_x a_0 \nabla_x u^\ep_0  = f - \gamma(\tfrac{x}{\ep}) v_\ep , \qquad t\in \RR, \, x \in \RR^d.\\
\end{equation}
Similarly, $v_\ep$ solves the heat equation 
\begin{equation}\label{te.e6}
\d_t v_\ep - \div_x\big( b(\tfrac{x}{\ep}) \nabla_x v_\ep \big)  = g+\overline{\gamma}(\tfrac{x}{\ep}) \d_t u_\ep,\qquad t\in \RR, \, x \in \RR^d,
\end{equation}
and Theorem \ref{evte} gives
\[
\| \Dt u_\ep \|_{L^2_\nu(\RR;L^2(\RR^d))} \lesssim \| \Dt f \|_{L^2_\nu(\RR;L^2(\RR^d))} +\|  g \|_{L^2_\nu(\RR;L^2(\RR^d))}.
\]
So, by Theorem \ref{hhomthm},  one has
	\begin{align}
	\| \Dt^{1/2} v_\ep - \Dt^{1/2} v^\ep_0 \|_{L^2_\nu(\RR;L^2(\RR^d))} &\le C \ep \big(  \| \Dt f \|_{L^2_\nu(\RR;L^2(\RR^d))} +\|  g \|_{L^2_\nu(\RR;L^2(\RR^d))} \big) \label{te.7}\\
	\| b(\tfrac{\cdot}{\ep}) \nabla v_\ep  - b(\tfrac{\cdot}{\ep} )\nabla ( v^\ep_0 + \ep N_0(\tfrac{\cdot}{\ep}) \cdot  \nabla  v^\ep_0)  \|_{L^2_\nu(\RR;L^2(\RR^d;\RR^d))} &\le C \ep \big(  \| \Dt f \|_{L^2_\nu(\RR;L^2(\RR^d))} +\|  g \|_{L^2_\nu(\RR;L^2(\RR^d))} \big)\label{te.e8} \\
	\| b(\tfrac{\cdot}{\ep}) \nabla v_\ep  - b_0   \nabla  v^\ep_0 \|_{L^2_\nu(\RR;H^{-1}(\RR^d;\RR^d))}& \le C \ep \big(  \| \Dt f \|_{L^2_\nu(\RR;L^2(\RR^d))} +\|  g \|_{L^2_\nu(\RR;L^2(\RR^d))} \big)\label{te.e9} 
\end{align}
for $v^\ep_0$ the solution to 
\begin{equation}\label{te.e10}
	\d_t v^\ep_0 - \div_xb_0 \nabla_x v^\ep_0  = g+\overline{\gamma}(\tfrac{x}{\ep}) \d_t u_\ep,\qquad t\in \RR, \, x \in \RR^d.
\end{equation}
We see that $U^\ep_0 := (\Dt u^\ep_0, a_0 \nabla u^\ep_0, v^\ep_0, b_0 \nabla v^\ep_0)^\top$ solves the homogenised evolutionary thermo-elastic system
\[
(\partial_t M_0+N_0+A)U^\ep_0 = F + R_\ep,  \quad \text{where $R_\ep: =(-\gamma(\tfrac{x}{\ep}) v_\ep+ \langle \gamma \rangle v^\ep_0 ,0, \overline{\gamma}(\tfrac{x}{\ep}) \d_t u_\ep - \overline{\langle \gamma \rangle} \Dt u^\ep_0 ,0)^\top$.}
\]
We shall prove below that 
\begin{equation}\label{te.e10.1}
\| (R_\ep)_1 \|_{H^1_\nu(\RR^d;L^2(\RR^d))}+\| (R_\ep)_2 \|_{L^2_\nu(\RR^d;L^2(\RR^d))} \lesssim \ep\big(  \| \Dt^2 f \|_{L^2_\nu(\RR;L^2(\RR^d))} +\| \Dt g \|_{L^2_\nu(\RR;L^2(\RR^d))} \big).
\end{equation}
This informs us that $U^\ep - U^0$ solves the homogenised evolutionary system with $\ep$-small right-hand side. Indeed, Theorem \ref{evte} implies
\begin{multline}\label{te.e10.2}
	\| \Dt u^\ep_0 - \Dt u_0  \|_{H_\nu^1(\RR;L^2(\RR^d))} + \| a_0 \nabla_x u_\ep - a_0 \nabla_x u_0 \|_{H_\nu^1(\RR;L^2(\RR^d ;\CC^d))\cap L^2_{\nu}(\RR; H_{\rm div}(\RR^d))} \\
	+ \|v_\ep - v_0 \|_{H_\nu^1(\RR;L^2(\RR^d))} + \|b_0 \nabla_x v_\ep - b_0 \nabla_x v_0  \|_{ L^2_{\nu}(\RR;  H_{\rm div}(\RR^d))} \\ \lesssim\ep\big(  \| \Dt^2 f \|_{L^2_\nu(\RR;L^2(\RR^d))} +\| \Dt g \|_{L^2_\nu(\RR;L^2(\RR^d))} \big).
\end{multline}
Inequalities \eqref{te.e2}-\eqref{te.e4}, \eqref{te.7}-\eqref{te.e9} and \eqref{te.e10.2} prove the desired result \eqref{temain1}.

The proof is complete, subject to establishing \eqref{te.e10.1}. Let us prove this inequality here.
Clearly, $(R_\ep)_1 =- \gamma(\tfrac{x}{\ep}) v^\ep_0$ up to a term of order $\ep$ that we have the desired control over (see \eqref{te.7}). Now,
 by Proposition \ref{mv} 
\[
\mathcal{P_\ep} \gamma(\tfrac{x}{\ep}) \mathcal{P}_\ep v^\ep_0 = \langle \gamma \rangle \mathcal{P}_\ep v^\ep_0,
\]
and we compute
\[
\langle \gamma \rangle  v^\ep_0 = \gamma(\tfrac{x}{\ep})  v^\ep_0 + r_\ep; \quad r_\ep : =  - (I-\mathcal{P_\ep}) \gamma(\tfrac{x}{\ep})  v^\ep_0  - \mathcal{P_\ep} \gamma(\tfrac{x}{\ep}) (I-\mathcal{P}_\ep) v^\ep_0  + \langle \gamma \rangle (I-\mathcal{P}_\ep) v^\ep_0.
\]
Inequality \eqref{Pe1}, the assumption that $\gamma(\tfrac{x}{\ep})$ is uniformly bounded and a uniform $H^1$-multiplier, see \eqref{multiplerest2},  and Theorem \ref{evte} give 
\begin{multline*}
\|\gamma(\tfrac{x}{\ep})  v^\ep_0 - \langle \gamma \rangle  v^\ep_0   \|_{L^2_\nu(\RR^d;L^2(\RR^d))}  = \|r_\ep \|_{L^2_\nu(\RR^d;L^2(\RR^d))}  \\\lesssim \ep  \|v^\ep_0 \|_{L^2_\nu(\RR^d;H^1(\RR^d))} \lesssim  \ep \big(  \| \Dt f \|_{L^2_\nu(\RR;L^2(\RR^d))} +\|  g \|_{L^2_\nu(\RR;L^2(\RR^d))} \big).
\end{multline*}
Then the desired inequality for $(R_\ep)_1$, in \eqref{te.e10.1}, follows.
Similarly,
\begin{multline*}
\|\overline{\gamma}(\tfrac{x}{\ep})  \Dt u^\ep_0 - \overline{\langle \gamma \rangle}  \Dt u^\ep_0   \|_{L^2_\nu(\RR^d;L^2(\RR^d))} \\ \lesssim \ep  \|\Dt u^\ep_0 \|_{L^2_\nu(\RR^d;H^1(\RR^d))} \lesssim  \ep \big(  \| \Dt f \|_{L^2_\nu(\RR;L^2(\RR^d))} +\|  g \|_{L^2_\nu(\RR;L^2(\RR^d))} \big),
\end{multline*}
which along with \eqref{te.e2} provides the desired inequality for $(R_\ep)_2$. That is, \eqref{te.e10.1} holds and the proof is complete. 
\end{proof}


\begin{thebibliography}{9}
\bibitem{BiSol} Birman, Michael Sh. and Solomjak, M.Z. {\it Spectral Theory of self-adjoint operators in Hilbert spaces}, first edition. Published by D. Reidel Publishing Company
P.O. Box 17,3300 AA Dordrecht, Holland, 1987.

\bibitem{BiSu04}{Birman, M. Sh., and Suslina, T. A., 2004. Second order periodic differential operators. Threshold properties and homogenisation. St. Petersburg. Math. J. 15(5), pp. 639-714.}

\bibitem{CoWa19}{ Cooper, S. and Waurick. M. {\it Fibre Homogenisation.} Journal of Functional Analysis, (2019) Volume 276, Issue 11 Pages 3363-3405.}

\bibitem{P09}{Picard, R. {\it A structural observation for linear material laws in classical mathematical physics} Mathematical Methods in the Applied Sciences, 32(14):1768--1803,2009}

\bibitem{PTW17}
Picard, Rainer; Trostorff, Sascha; Waurick, Marcus
{\it On maximal regularity for a class of evolutionary equations.}
J. Math. Anal. Appl. 449, No. 2, 1368-1381 (2017).

\bibitem{STW22}
C.~Seifert, S.~Trostorff, and M.~Waurick
\newblock {\em Evolutionary Equations}
\newblock {Birkh\"auser, Cham, 2022}

\bibitem{W13}
Waurick M.  
{\it Homogenization of a class of linear partial differential equations}
Asymptotic Analysis 82: 271-294, 2013 


\bibitem{Su13}{Suslina, T.A. {\it Homogenization of a stationary periodic Maxwell system}, St. Petersburg	Math. J. 16:5 (2005), 863–922}

\bibitem{TrWa21}{Trostorff, S., Waurick, M. {\it Maximal Regularity for Non-autonomous Evolutionary Equations.} Integr. Equ. Oper. Theory 93, 30 (2021). https://doi.org/10.1007/s00020-021-02645-5
}

\bibitem{Zh89}{Zhikov, V. V., 1989. Spectral approach to asymptotic problems in diffusion. Differ. Equ. 25, pp. 33-39.}

\bibitem{ZhPa}{Zhikov, V. V., Pastukhova, S. E., 2016. Operator estimates in homogenization theory. Russian Mathematical Surveys, 71 417.} 
\end{thebibliography}
\end{document}